\newif\ifarxiv
\arxivtrue
\ifarxiv
  \documentclass[preprint,12pt]{elsarticle}
\else
  \documentclass[preprint,12pt]{elsarticle}
  \usepackage{lineno}
  \modulolinenumbers[5]
\fi
\journal{Computer Methods in Applied Mechanics and Engineering}
\usepackage[utf8]{inputenc} 
\usepackage[english]{babel}
\usepackage[T1]{fontenc}   
\usepackage[normalem]{ulem}
\usepackage{mathpazo}
\usepackage{graphicx}
\usepackage{comment}
\usepackage{amsfonts,stmaryrd,mathrsfs}
\usepackage{bbm}
\usepackage{bm}
\usepackage{enumitem}
\usepackage{tikz-cd}
\usepackage{tikz}
\usepackage{makecell}
\usetikzlibrary{positioning,shapes.geometric,arrows.meta,calc}
\usepackage[top=2.5cm, bottom=2.5cm, left=2cm, right=2cm]{geometry}
\usepackage{subfig}
\usepackage{multirow}
\usepackage{physics}
\usepackage{array}
\usepackage{bigints}
\usepackage{etoolbox}
\usepackage{mdframed}
\usepackage{color}
\usepackage{booktabs}

\usepackage{tabularx}
\usepackage[unicode,colorlinks=true,linkcolor=blue,citecolor=red,pdfencoding=auto,psdextra]{hyperref}
 \usepackage{dsfont}

\usepackage{float}

\usepackage{amsmath}
\usepackage{upgreek}
\usepackage{commath}
\usepackage{amssymb}
\usepackage{xcolor}
\usepackage{amsthm}
\usepackage{csquotes}
\usepackage{titlesec}

\DeclareRobustCommand{\R}{\mathbb{R}}

\newtheorem{theorem}{Theorem}
\newtheorem{lemma}[theorem]{Lemma}

\newmdtheoremenv{proposition}[theorem]{Proposition}

\newtheorem{remark}[theorem]{Remark}

\titleformat{\subsubsection}[runin]
{\bfseries}{\thesubsubsection}{1em}{}

\makeatletter
\renewcommand\subsubsection{\@startsection{subsubsection}{3}{\z@}%
  {12\p@ \@plus 6\p@ \@minus 3\p@}%
  {-0.6em}%
  {\normalfont\normalsize\bfseries}}
\makeatother

\begin{document}
\begin{frontmatter}
\title{AMOR: Adaptive Modal Operator Refinement for Operator Learning on Arbitrary Meshes and Irregular 3D Geometries}

\cortext[cor1]{Corresponding author}
\author[lbnl,ethz]{David Voderholzer\fnref{equal}}
\ead{dvoderholzer@lbl.gov}
\author[kit]{Leonard Vetter\fnref{equal}\corref{cor1}}
\ead{vetter@kit.edu}
\author[osu]{Thomas O'Leary-Roseberry}
\ead{oleary-roseberry.1@osu.edu}
\fntext[equal]{These authors contributed equally to this work.}
\affiliation[lbnl]{organization={Lawrence Berkeley National Laboratory}, city={Berkeley}, state={CA}, country={USA}}
\affiliation[ethz]{organization={ETH Z\"urich}, city={Z\"urich}, country={Switzerland}}
\affiliation[kit]{organization={Karlsruhe Institute of Technology}, city={Karlsruhe}, country={Germany}}
\affiliation[osu]{organization={The Ohio State University}, city={Columbus}, state={OH}, country={USA}}

\begin{abstract}
Parametric PDEs in science and engineering are often posed on irregular meshes with geometries that may contain holes or complex boundaries, yet popular operator-surrogate architectures such as the Fourier Neural Operator (FNO) rely on structured geometries for fast computational kernels. 
We introduce Adaptive Modal Operator Refinement (AMOR), which replaces FNO's Fourier-based spectral convolutions with reduced-basis blocks and supports both data-driven and prescribed input/output bases through a common encode/decode interface.
AMOR thus applies unmodified to arbitrary meshes and, because its spectral parameterization has no tensor-product dependence on the spatial dimension, it scales favorably to three dimensions, where FNO's spectral weights grow as $m^3$ with the number of spectral modes $m$. 
We prove that AMOR is a universal approximator of continuous operators.
On two-dimensional structured-grid benchmarks, AMOR is competitive with FNO overall and surpasses PCANet: it outperforms FNO on nonlinear diffusion and Helmholtz problems, while a static snapshot basis keeps it from matching FNO's accuracy on the time-dependent Kolmogorov flow problem.
On three-dimensional problems it matches or exceeds FNO's accuracy at lower inference cost on linear elastostatics while trailing FNO on a turbulent lid-driven-cavity flow and scales favorably to problems demanding high capacity.
On irregular meshes, where FNO is not directly applicable, AMOR reduces the relative $L^2$ error over the FNO extension Geo-FNO by $31\%$ (airfoil) and $42\%$ (elasticity).
Finally, AMOR maintains strong performance across three physically distinct, experimentally motivated loading regimes on a topologically nontrivial three-dimensional aircraft-rib mesh, demonstrating its applicability to realistic engineering problems without modification.
\end{abstract}


\end{frontmatter}

\section{Introduction}\label{sec:intro}
Solving parametric partial differential equations (PDEs) on arbitrary geometries is central to the design and analysis of engineering systems.
Finite-element and finite-volume methods re-solve the discretized system for every parameter instance.
For $N_1$ degrees of freedom per spatial direction in $d$ dimensions, the total number of unknowns grows as $\mathcal{O}(N_1^d)$, with the cost of a single spatial solve scaling at least linearly and often superlinearly in this number. For time-dependent problems, this cost is incurred repeatedly over the time integration, introducing an additional scaling with the number of time steps.
Consequently, many-query tasks such as uncertainty quantification, optimization, and real-time prediction become expensive~\cite{evans2010partial,cohen2015approximation,kutyniok2022theoretical,bellman1957dynamic}.
Neural operators~\cite{kovachki2023neural,chen1995universal, lu2021learning, li2020neural, li2020multipole} provide a framework for constructing efficient surrogate models by learning mappings between function space representations of PDE inputs and outputs.
Unlike conventional surrogate models tied to a particular discretization, neural operators are formulated on the operator level, with the goal of supporting evaluation across parameter values and discretization resolutions, once trained. 
For PDE solution operators possessing intrinsic low-dimensional structure, neural operators can provably overcome the curse of dimensionality, achieving approximation rates that depend on the regularity of the solution manifold rather than the discretization resolution~\cite{kovachki2021universal, lanthaler2022error,herrmann2024neural}.

The Fourier Neural Operator (FNO)~\cite{li2021fourier} achieves resolution-invariant, globally receptive predictions by learning convolutions in Fourier space~\cite{duruisseaux2025fourier}, but its standard FFT implementation assumes an equispaced Cartesian grid --- the exception rather than the rule for real engineering geometries~\cite{baker2005mesh} --- for the Fast Fourier Transform (FFT)~\cite{cooley1965algorithm} computation.
The same tensor-product structure makes its spectral weights scale as $\mathcal{O}(m^{d})$ for $m$ retained modes per axis in $d$ dimensions, inflating both the parameter count and the cost of the spectral multiplication itself; in three dimensions this is not only a memory bottleneck but a measurable compute-time one. DeepONet~\cite{lu2021learning} evaluates outputs at arbitrary points but uses a fixed sensor representation for its inputs; graph-based operators act on irregular discretizations directly. 
Reduced-basis methods such as PCANet~\cite{bhattacharya2021model}, POD-DeepONet~\cite{lu2022comprehensive}, POD-DL-ROM~\cite{fresca2022pod}, and derivative-informed projected neural networks~\cite{o2022derivative} operate on low-dimensional basis coefficients but do not interleave pointwise nonlinearities in physical space, a key feature of FNO-type models.

Industrial simulation data rarely live on uniform grids: meshes generated from CAD geometry are refined at holes, fillets, contacts, and boundary layers, so element sizes vary by orders of magnitude. 
Resampling onto a uniform grid sets the spacing by the smallest feature in the bounding box, replaces boundaries by staircases or masks, and requires interpolating predictions back to the nodes on which loads, constraints, and peak stresses are defined; learned deformations to a latent grid~\cite{li2023fourier} and bounding-box embeddings~\cite{liu2023dafno} reduce but do not remove these costs. 

We introduce Adaptive Modal Operator Refinement (AMOR), a novel operator-learning architecture that combines the depth and compositional flexibility of neural operators with the geometric adaptability of modal representations, replacing FNO's Fourier-based spectral convolution with reduced-basis blocks usable on arbitrary meshes. Our contributions are:
\begin{itemize}
\item[(C1)] \textbf{Mesh-agnostic, adaptive cross-basis operator.} AMOR replaces FNO's FFT-based spectral convolution with reduced-basis blocks evaluated directly on mesh nodes, using separate input and output reduced bases coupled by a dense learned transformation across all retained modes. 
This makes AMOR agnostic to mesh structure -- it operates on arbitrary node sets without requiring a structured grid -- and admits any fixed or learned basis through a common encode/decode interface without changing the architecture.

\item[(C2)] \textbf{Universality without coordinate inputs.} We prove that AMOR with continuous, point-separating bases that contain the constant function approximates any continuous operator uniformly on compact sets of bounded inputs. The proof extends the averaging argument of~\cite{lanthaler2025nonlocality}: the cross-mode coupling decodes the modes from a constant channel and thereby supplies the positional encoding that the componentwise lifting and projection lack.

\item[(C3)] \textbf{Scalable reduced representation for large 3D problems.} 
By parametrizing the nonlocal operator in reduced coefficient space, rather than in the Fourier spectral product representation,
AMOR avoids the explicit tensor-product mode scaling of Fourier Neural operators. 
We demonstrate competitive accuracy and favorable computational scaling on two- and three-dimensional benchmarks, including problems on highly irregular meshes.

\end{itemize}

\par \textbf{Related Work.}
Truncating the Fourier expansion to a fixed number of modes imposes an implicit smoothness bias, under-resolving high-frequency features and sharp gradients when a solution's energy is not concentrated in the low-frequency spectrum~\cite{duruisseaux2025fourier}; this has motivated both factorized spectral variants that reduce the $\mathcal{O}(m^d)$ parameter cost, e.g.\ F-FNO~\cite{tran2023factorized}, and a re-examination of whether the Fourier basis itself is optimal for spectral operator learning.
PODNO~\cite{cheng2026podno} shows that replacing the Fourier transform with a data-driven POD basis improves spectral efficiency when solution manifolds are not well represented by low Fourier modes, while retaining the advantages of spectral neural operators. 
However, it shares a single POD basis between inputs and outputs -- restrictive when the two fields are physically dissimilar (e.g. a coefficient field and a displacement field, or scalar and vector-valued quantities) -- and, like FNO, is presented only on structured grids. 
AMOR generalizes this idea to separate input and output bases, vector-valued operators, and non-tensor-product meshes, with repeated reduced-basis blocks interleaved with physical-space nonlinearities. 
Other operator architectures replace Fourier spectral layers with alternative convolutional or transform-based representations: convolutional kernels (CNO)~\cite{raonic2023convolutionalneuraloperatorsrobust}, wavelets (WNO)~\cite{TRIPURA2023115783}, Laplace-domain representations (LNO)~\cite{cao2023lnolaplaceneuraloperator}, spherical-harmonic transforms on manifolds (SFNO)~\cite{bonev2023sphericalfourierneuraloperators}, Laplace--Beltrami eigenfunctions (NORM)~\cite{chen2024norm,fano2026}, graph-Laplacian eigenbases which extend spectral layers to unstructured meshes~\cite{sarkar2025sp2gno,behroozi2026gsno}, and learned adaptive frames (ABLE)~\cite{zhao2026able}.

FNO and its variants have been applied to weather forecasting~\cite{pathak2022fourcastnet}, subsurface flow~\cite{wen2022ufno}, and turbulence modeling~\cite{li2021fourier, kochkov2021machine}, but all on structured grids, whereas aerodynamics on complex shapes~\cite{elrefaie2024drivaernet, elrefaie2024drivaernetpp}, patient-specific hemodynamics~\cite{yin2022simulating}, microstructure-resolved materials~\cite{rashid2022learning}, and seismic wave propagation~\cite{yang2021seismic} typically require unstructured meshes. Existing remedies either keep the Fourier basis and adapt the domain to it -- Geo-FNO~\cite{li2023fourier} learns a diffeomorphism to a latent uniform grid (degrading for complex topologies), direct spectral evaluation (DSE)~\cite{lingsch2024beyond} evaluates Fourier modes on non-equispaced points at $\mathcal{O}(mN)$ cost, and DAFNO~\cite{liu2023dafno} embeds a smoothed characteristic function -- retaining FFT-based computation on an embedding grid; or replace the spectral basis with pointwise or mesh-free representations: graph neural operators~\cite{li2020neural, li2023gino, alkin2024upt, alkin2025abupt, hao2023gnot, wu2024transolver, wen2025geometry, mousavi2025rignographbasedframeworkrobust} and mesh-based simulators~\cite{pfaff2021learningmeshbasedsimulationgraph, pmlr-v119-sanchez-gonzalez20a} act on unstructured meshes via message passing, with cost scaling with graph edges; mesh-informed neural networks restrict dense layers to local connections between finite-element spaces~\cite{franco2023minn}; and coordinate/INR-based models (CORAL~\cite{serrano2023operator}, AROMA~\cite{NEURIPS2024_185a120a}, PI-GANO~\cite{zhong2025pigano}) gain geometric flexibility at the price of inference efficiency or a physics-informed training requirement.

\section{Mathematical Setup and Theory}

Let $\Omega\subset\mathbb R^d$ be a bounded Lipschitz domain of positive measure, and set
$\mathcal U=L^2(\Omega;\mathbb R^{C_{\mathrm{in}}})$ and
$\mathcal V=L^2(\Omega;\mathbb R^{C_{\mathrm{out}}})$. For either component count $C$, the inner product is
$\langle f,g\rangle=\sum_{c=1}^C\int_\Omega f_c(x)g_c(x)\,\mathrm dx$.
We seek to learn an operator $\Psi^\dagger\colon\mathcal U\to\mathcal V$ from
input--output pairs $\{(a_i,u_i)\}_{i=1}^{n_s}$, with $u_i=\Psi^\dagger(a_i)$.
A typical example is a PDE solution map defined implicitly by
\begin{equation}\label{eq:parametric_pde}
    \mathcal N(u,a)=0.
\end{equation}

We call finite orthonormal systems
$\{\varphi_m^{\mathrm{in}}\}_{m=1}^{k_{\mathrm{in}}}\subset\mathcal U$ and
$\{\varphi_n^{\mathrm{out}}\}_{n=1}^{k_{\mathrm{out}}}\subset\mathcal V$
\emph{reduced bases}. Their elements are modes, and
$\langle a,\varphi_m^{\mathrm{in}}\rangle$ is a modal coefficient.
The input and output bases may be chosen independently. Section~\ref{subsec:bases_taxonomy} describes data-adaptive (POD, learned), derivative-informed (DIS), geometry-intrinsic (Laplacian), and prescribed (polynomial) constructions.

\subsection{Operator Learning Framework}\label{subsec:operator_learning}
We use the nonlocal neural operator framework of
Lanthaler et al.~\cite{lanthaler2025nonlocality}, with separate component and channel indices.
A hidden state has $C$ physical components and $d_v$ channels per component; each channel
is a vector-valued function on $\Omega$. A network has the form
\begin{equation}\label{eq:neural_operator}
    \Psi_\theta=\mathcal Q\circ\mathcal H^{(L)}\circ\cdots\circ
    \mathcal H^{(0)}\circ\mathcal R,
    \qquad L\ge1.
\end{equation}
Thus there are $L+1$ hidden operator layers. 
Theorem~\ref{thm:AMOR_universality} uses $L\ge2$.
The lifting and projection act identically in every component:
\begin{equation}\label{eq:lifting}
\begin{aligned}
    \mathcal R(a)(x)_{c,j}&=w^{\mathcal R}_j a_c(x)+b^{\mathcal R}_j,\\
    \mathcal Q(h)(x)_c&=W^{\mathcal Q}_2\sigma\bigl(W^{\mathcal Q}_1h_{c,\cdot}(x)+b^{\mathcal Q}_1\bigr)+b^{\mathcal Q}_2,
\end{aligned}
\end{equation}
where $w^{\mathcal R},b^{\mathcal R}\in\mathbb R^{d_v}$,
$W^{\mathcal Q}_1\in\mathbb R^{d_q\times d_v}$, $b^{\mathcal Q}_1\in\mathbb R^{d_q}$,
$W^{\mathcal Q}_2\in\mathbb R^{1\times d_q}$, and $b^{\mathcal Q}_2\in\mathbb R$.
The channel width $d_v$ and the projection width $d_q$ may both vary with the desired approximation accuracy.
Our implementation uses $d_q=d_v$.

The component counts are
\begin{equation}\label{eq:layer_widths}
    C^{(0)}=C^{(1)}=C_{\mathrm{in}},\qquad
    C^{(\ell)}=C_{\mathrm{out}}\quad(2\le\ell\le L+1).
\end{equation}
A standard nonlocal hidden layer is
\begin{equation}\label{eq:hidden_layer}
\begin{aligned}
    h^{(\ell+1)}(x)&=\sigma\bigl(W^{(\ell)}h^{(\ell)}(x)
             +(\mathcal K^{(\ell)}h^{(\ell)})(x)+b^{(\ell)}\bigr),\\
    (\mathcal K^{(\ell)}h)(x)&=\int_\Omega\kappa^{(\ell)}(x,y)h(y)\,\mathrm dy.
\end{aligned}
\end{equation}
The activation acts entrywise. Except at the transition layer $\ell=1$, the pointwise
linear maps and biases act on channels only:
\begin{equation}\label{eq:pointwise_standard}
    W^{(\ell)}=\operatorname{Id}_{C^{(\ell)}}\otimes\widetilde W^{(\ell)},
    \qquad b^{(\ell)}_{c,j}=\widetilde b^{(\ell)}_j,
    \qquad
    \widetilde W^{(\ell)}\in\mathbb R^{d_v\times d_v},\quad
    \widetilde b^{(\ell)}\in\mathbb R^{d_v}.
\end{equation}
Here $\otimes$ denotes the Kronecker product, so $\widetilde W^{(\ell)}$ and $\widetilde b^{(\ell)}$ act identically on every component. At $\ell=1$, the map
$W^{(1)}\colon\mathbb R^{C_{\mathrm{in}}\times d_v}\to
\mathbb R^{C_{\mathrm{out}}\times d_v}$ and the bias
$b^{(1)}\in\mathbb R^{C_{\mathrm{out}}\times d_v}$ are unrestricted.
The transition follows one layer of processing in the input basis.

Throughout, we impose
\begin{equation}\label{eq:sigma_assump}
\begin{gathered}
    \sigma\colon\mathbb R\to\mathbb R\ \text{is globally Lipschitz and nonpolynomial},
    \qquad \sigma(0)=0,\\
    \sigma\ \text{is strictly monotone on a nonempty open interval }\mathcal I.
\end{gathered}
\end{equation}
ReLU, GELU, and tanh satisfy these assumptions; unlike \cite{lanthaler2025nonlocality}, we do not require $\sigma$ to be smooth. 
Global Lipschitz continuity makes the pointwise activation a continuous map on $L^2$; local monotonicity is used only in the universality construction.

\subsection{AMOR as a Neural Operator}\label{subsec:amor_as_no}
Adaptive Modal Operator Refinement (AMOR) uses finite-rank integral operators expressed
in the reduced bases. The decoder and encoder for the input basis, applied columnwise, are
\begin{equation}\label{eq:decoding_operator}
    (\bm\varphi^{\mathrm{in}}Z)_{\cdot,j}
       =\sum_{m=1}^{k_{\mathrm{in}}}Z_{m,j}\varphi_m^{\mathrm{in}},
    \qquad
    \bigl((\bm\varphi^{\mathrm{in}})^*h\bigr)_{m,j}
       =\langle h_{\cdot,j},\varphi_m^{\mathrm{in}}\rangle.
\end{equation}
They map $\mathbb R^{k_{\mathrm{in}}\times d_v}$ to $\mathcal U^{d_v}$ and back,
and are mutually adjoint. The output operators are defined analogously.
In particular, $\Pi_{\mathrm{in}}=\bm\varphi^{\mathrm{in}}(\bm\varphi^{\mathrm{in}})^*$
and $\Pi_{\mathrm{out}}=\bm\varphi^{\mathrm{out}}(\bm\varphi^{\mathrm{out}})^*$
are orthogonal projections onto the corresponding reduced spaces.
The layer bases are
\begin{equation}\label{eq:block_bases}
    (\bm\varphi_{\mathrm{enc}}^{(\ell)},\bm\varphi_{\mathrm{dec}}^{(\ell)})=
    \begin{cases}
      (\bm\varphi^{\mathrm{in}},\bm\varphi^{\mathrm{in}}),&\ell=0,\\
      (\bm\varphi^{\mathrm{in}},\bm\varphi^{\mathrm{out}}),&\ell=1,\\
      (\bm\varphi^{\mathrm{out}},\bm\varphi^{\mathrm{out}}),&\ell\ge2.
    \end{cases}
\end{equation}
Suppressing the layer index on the mode counts, the reduced-basis integral operator is
\begin{equation}\label{eq:AMOR_layer}
    \mathcal K_{\mathrm{RB}}^{(\ell)}h
      =\bm\varphi_{\mathrm{dec}}^{(\ell)}
          \bigl[T^{(\ell)}(\bm\varphi_{\mathrm{enc}}^{(\ell)})^*h\bigr],
    \qquad
    (TZ)_{n,i}=\sum_{m=1}^{k_{\mathrm{enc}}}\sum_{j=1}^{d_v}T_{n,m,i,j}Z_{m,j},
\end{equation}
where $T^{(\ell)}\in\mathbb R^{k_{\mathrm{dec}}\times k_{\mathrm{enc}}\times d_v\times d_v}$.
Its kernel is
\begin{equation}\label{eq:rb_kernel}
    \kappa^{(\ell)}(x,y)_{(c,i),(c',j)}
      =\sum_{n,m}T^{(\ell)}_{n,m,i,j}
          \varphi^{\mathrm{dec}}_{n,c}(x)\varphi^{\mathrm{enc}}_{m,c'}(y).
\end{equation}
This finite-rank kernel is in general neither translation invariant nor a convolution kernel.
\begin{remark}[Translation equivariance]\label{rem:equivariance}
FNO multiplies each retained Fourier mode by its own matrix (\ref{si:fno}), which makes it translation equivariant on periodic domains. The dense coupling $T^{(\ell)}$, data-adapted modes, and reference offsets remove this property from AMOR, so a translation symmetry of $\Psi^\dagger$, where present, must be learned from data. On bounded domains with boundary conditions, localized loads, or geometric features, as in most engineering applications, $\Psi^\dagger$ has no such symmetry, and AMOR represents its position dependence without coordinate inputs (Section~\ref{subsec:universality}).
\end{remark}

AMOR uses gated residual blocks, with no residual connection across the transition:
\begin{equation}\label{eq:AMOR_block}
\begin{aligned}
 h^{(\ell+1)}&=h^{(\ell)}+
   \sigma\bigl(g_\ell[W^{(\ell)}h^{(\ell)}
             +\mathcal K_{\mathrm{RB}}^{(\ell)}h^{(\ell)}+b^{(\ell)}]\bigr),
       &&\ell\ne1,\\
 h^{(2)}&=\sigma\bigl(g_1[W^{(1)}h^{(1)}
             +\mathcal K_{\mathrm{RB}}^{(1)}h^{(1)}+b^{(1)}]\bigr).
\end{aligned}
\end{equation}
Here $g_\ell\in\mathbb R$ are learnable scalar gates.
The encoders and decoders are bounded linear maps. The affine lifting and the
shallow projection are pointwise Lipschitz, and \eqref{eq:sigma_assump} makes each
hidden block Lipschitz on its $L^2$ state space. Consequently, the complete network
$\Psi_\theta$ is a continuous map between the stated $L^2$ spaces.

\begin{remark}[Including a constant mode]\label{rem:constant_mode}
For a basis with $C$ components, let $e_{\mathbf1,C}=(C|\Omega|)^{-1/2}\mathbf1$. 
If necessary, enlarge the reduced space by $\operatorname{span}\{\mathbf1\}$ and choose an orthonormal basis whose first mode is $e_{\mathbf1,C}$; the dimension increases by at most one. 
Our experiments omit this augmentation because it did not improve accuracy.
\end{remark}

\begin{remark}[Reference offsets]\label{rem:offsets}
Our implementation uses fixed reference fields $\mu_{\mathrm{in}}\in\mathcal U$ and $\mu_{\mathrm{out}}\in\mathcal V$, the training-split means, and replaces the kernel path in every block and column by
\[
 \mathcal K^{\mu,(\ell)}_{\mathrm{RB}}h=\bm\varphi^{(\ell)}_{\mathrm{dec}}\,T^{(\ell)}
 (\bm\varphi^{(\ell)}_{\mathrm{enc}})^*(h-\mu_{\mathrm{enc}})+\mu_{\mathrm{dec}} .
\]
Its linear part is \eqref{eq:AMOR_layer}; the remainder is a fixed spatial bias for fixed $T^{(\ell)}$. Offsets are not needed for universality: Theorem~\ref{thm:AMOR_universality} is proved with zero offsets and extends to this affine variant (Remark~\ref{rem:offsets_univ}).
\end{remark}

\subsection{Universal Approximation}\label{subsec:universality}

AMOR applies its reduced-basis projection only inside the integral path.
Its hidden states and nonlinear outputs need not lie in the reduced spaces.
Fixed, finite bases therefore suffice for arbitrarily accurate approximation under the conditions below.

Set $X_C=\bar\Omega\times\{1,\ldots,C\}$, and abbreviate the input and output
versions by $X_{\mathrm{in}}$ and $X_{\mathrm{out}}$.
For (reduced) bases with continuous representatives on $\bar\Omega$, define
\begin{equation}\label{eq:mode_map}
 \mathbf p^{\mathrm{in}}(y,c)
   =\bigl(\varphi^{\mathrm{in}}_{m,c}(y)\bigr)_{m=1}^{k_{\mathrm{in}}},
 \qquad
 \mathbf p^{\mathrm{out}}(x,c)
   =\bigl(\varphi^{\mathrm{out}}_{n,c}(x)\bigr)_{n=1}^{k_{\mathrm{out}}}.
\end{equation}
A basis \emph{separates points} if its mode-value map is injective,
including across different components.

\begin{theorem}[Universality of AMOR with fixed bases]\label{thm:AMOR_universality}
Assume \eqref{eq:sigma_assump}. Let $K\subset\mathcal U$ be compact in
$L^2$, with
$\sup_{a\in K}\max_{1\le c\le C_{\mathrm{in}}}
\|a_c\|_{L^\infty(\Omega)}\le M_\infty<\infty$,
and let $\Psi^\dagger\colon K\to\mathcal V$ be a continuous operator.
Fix finite continuous orthonormal input and output bases that both separate
points and contain their respective normalized constant modes
$(C|\Omega|)^{-1/2}\mathbf1$.
For every $\varepsilon>0$, there is an AMOR network $\Psi_\theta$ with these bases, any depth $L\ge2$, and finite widths $d_v,d_q$, such that
\[
 \sup_{a\in K}\|\Psi^\dagger(a)-\Psi_\theta(a)\|_{\mathcal V}
 <\varepsilon.
\]
The widths may also be chosen equal, $d_q=d_v$.
\end{theorem}

The proof builds on approximation by averaged local features
\cite[Lemma~A.9, Remark~A.10]{lanthaler2025nonlocality}.
AMOR supplies position internally without explicitly adding coordinates: off-diagonal modal couplings turn a
constant channel into the separating mode values. The proof and the extension to the factorization in
Remark~\ref{rem:factorized} are given in~\ref{si:universality}.

\ref{si:universality_scope} discusses which bases satisfy these assumptions, and~\ref{si:universality_comparison} compares the result with those for FNO, PODNO, and PCA-Net. 
For example, for scalar fields, any span containing $1,x_1,\ldots,x_d$ separates
points.

\subsection{AMOR's Architecture and Implementation}\label{subsec:amor_architecture}
For a discretization with $N$ spatial degrees of freedom shared by the input and output,
a $C$-component field is represented by a vector in $\mathbb R^D$, $D=CN$.
A hidden state is a matrix $H^{(\ell)}\in\mathbb R^{D^{(\ell)}\times d_v}$.
For continuous inputs these degrees of freedom may be nodal values. For general
$L^2$ inputs, a well-defined discretization such as cell averages or finite-element
projection is required, because point evaluation is not defined on $L^2$ equivalence classes.
We use a symmetric positive-definite Gram matrix
\begin{equation}\label{eq:gram_structure}
    M=I_C\otimes M_x
\end{equation}
and the discrete inner product $\langle\underline u,\underline v\rangle_M=\underline u^\top M\underline v$, where $\underline u,\underline v\in\mathbb R^D$ are the coefficient vectors of discrete fields $u_h,v_h$ in a finite-dimensional space $\mathcal U_h\subset L^2(\Omega;\mathbb R^C)$. Underlined symbols denote coefficient vectors throughout. For the consistent finite-element Gram matrix, $\langle u_h,v_h\rangle_{L^2}=\underline u^\top M\underline v$; for lumped or equal-weight quadrature, the equality holds approximately.

\begin{remark}[Gram matrices]\label{rem:gram}
For equal-volume cells, the diagonal quadrature matrix is $M_x=\Delta V I_N$,
where $\Delta V$ is the cell volume. On unstructured meshes, one can use a consistent
finite-element Gram matrix or a lumped Gram matrix with the corresponding positive
volume weights. Equal weights on point clouds define an empirical inner product;
their consistency with the Lebesgue inner product requires suitable sampling or
quadrature assumptions. Boundary-node weights must follow the chosen quadrature rule.
We use equal weights $M_x=\Delta V I_N$ on all structured grids, lumped node-area weights on the Airfoil and Elasticity meshes, the consistent P1 mass matrix for the cantilever-tip rib model, and equal weights for the other rib models.
With equal weights, Euclidean PCA gives $\Phi_E=(\Delta V)^{1/2}\Phi$ and $\Phi_E T\Phi_E^\top=\Phi T\Phi^\top M$, so the block is unchanged.
\end{remark}

On the reference mesh, let $\Phi\in\mathbb R^{D\times k}$ contain the reduced
modes as columns. These bases are constructed to satisfy
\begin{equation}\label{eq:discrete_orthonormality}
    \Phi^*\Phi=\Phi^\top M\Phi=I_k,
    \qquad \Phi^*=\Phi^\top M.
\end{equation}
The reference encoder is therefore the $M$-adjoint of the decoder; for DIS input bases it is the $\mathcal C_{\mathrm{pr}}^{-1}$-adjoint (Section~\ref{subsec:bases_taxonomy}).
After mesh transfer, the decoder contains the prolonged reference modes and the
encoder is their mass-weighted dual, as defined in Section~\ref{subsec:meshes}.
The adjoint and the dual agree for an orthonormal basis; transported bases are not
reorthonormalized.

To cover both meshes, write $B_{\mathrm{enc}},B_{\mathrm{dec}}$ for the decoder
matrices and $E_{\mathrm{enc}}$ for the coefficient encoder. On the reference mesh,
$B=\Phi$ and $E=\Phi^\top M$. A block performs
\begin{equation}\label{eq:transform}
    Z=E_{\mathrm{enc}}H,
    \qquad Z'_{n,i}=\sum_{m,j}T_{n,m,i,j}Z_{m,j},
    \qquad \widetilde H=B_{\mathrm{dec}}Z',
\end{equation}
followed by \eqref{eq:AMOR_block}, with $\widetilde H$ replacing the integral term.
With offsets, this path becomes
$B_{\mathrm{dec}}[T E_{\mathrm{enc}}(H-\mu_{\mathrm{enc}})]+\mu_{\mathrm{dec}}$,
where each reference field is applied columnwise; all experiments use this affine path.

\begin{remark}[Factorized coefficient weights]\label{rem:factorized}
A dense tensor has $d_v^2k_{\mathrm{enc}}k_{\mathrm{dec}}$ parameters.
The factorization $T_{n,m,i,j}=\mathsf B_{n,m,i}\mathsf A_{m,i,j}$ gives
\begin{equation}\label{eq:factorized}
    Z'_{n,i}=\sum_m\mathsf B_{n,m,i}\sum_j\mathsf A_{m,i,j}Z_{m,j}.
\end{equation}
It first mixes channels within each encoding mode, then mixes modes within each
output channel. Its parameter count is $d_vk_{\mathrm{enc}}(d_v+k_{\mathrm{dec}})$,
a fraction $1/k_{\mathrm{dec}}+1/d_v$ of the dense count. This reduces the count
when that fraction is below one. The factorization restricts the general tensor
class, but retains the universal construction in Remark~\ref{rem:factorized_universality}.
\end{remark}

\textbf{Parameter conditioning and global skip.}
Scalar physical parameters $\bm\pi\in\mathbb R^{d_\pi}$ enter by FiLM~\cite{perez2017filmvisualreasoninggeneral}. 
One network with a hidden layer of width $d_{\mathrm{film}}$ maps $\bm\pi$ to $(\gamma,\beta)\in\mathbb R^{2d_v}$, and $\mathrm{FiLM}(h)_{c,j}=(1+\gamma_j)h_{c,j}+\beta_j$ acts identically in every component.
It is applied to the lifted state, to every block update,
\[
 h^{(\ell+1)}=h^{(\ell)}+\mathrm{FiLM}\bigl(\sigma(g_\ell[\,\cdot\,])\bigr)\ (\ell\ne1),\qquad
 h^{(2)}=\mathrm{FiLM}\bigl(\sigma(g_1[\,\cdot\,])\bigr),
\]
with $[\,\cdot\,]$ as in \eqref{eq:AMOR_block}, and to the hidden layer of the projection, $\mathcal Q(h)_c=W^{\mathcal Q}_2\,\mathrm{FiLM}\bigl(\sigma(W^{\mathcal Q}_1h_{c,\cdot}+b^{\mathcal Q}_1)\bigr)+b^{\mathcal Q}_2$, which requires $d_q=d_v$. 
The same $(\gamma,\beta)$ is used at all sites and adds $d_{\mathrm{film}}(d_\pi+2d_v+1)+2d_v$ parameters.
For $\gamma=\beta=0$, FiLM is the identity, so the conditioned class contains the networks of Theorem~\ref{thm:AMOR_universality}.
An optional global skip adds $g_{\mathcal S}(W^{\mathcal S}a+b^{\mathcal S})$, with $W^{\mathcal S}\in\mathbb R^{C_{\mathrm{out}}\times C_{\mathrm{in}}}$, $b^{\mathcal S}\in\mathbb R^{C_{\mathrm{out}}}$; $g_{\mathcal S}=0$ disables it.
Its use per dataset is listed in~\ref{si:film}.

\begin{figure}[htbp]
    \centering
    \includegraphics[width=\linewidth]{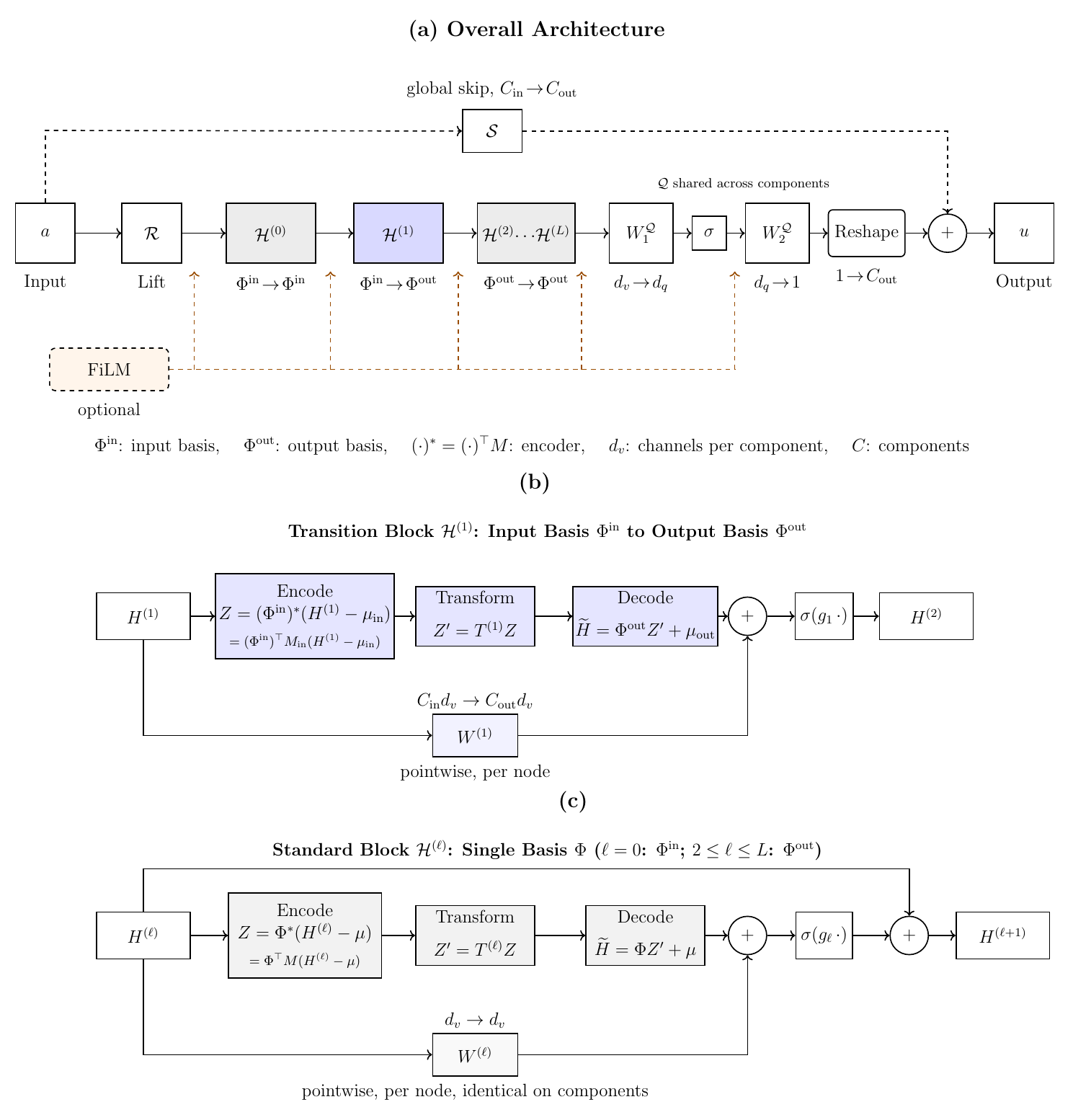}
    \caption{AMOR architecture. (a) Lifting $\mathcal R$, hidden operator layers
    $\mathcal H^{(0)},\ldots,\mathcal H^{(L)}$, the shared shallow projection
    $\mathcal Q$, and optional global skip and FiLM conditioning.
    (b) On the reference mesh, the transition block encodes in
    $\Phi^{\mathrm{in}}$, applies $T^{(1)}$, and decodes in $\Phi^{\mathrm{out}}$;
    on a target mesh these operations use the prolonged decoder and dual encoder
    from Section~\ref{subsec:meshes}. Its unrestricted pointwise map changes
    the component count. (c) Standard blocks use one basis on both sides and a
    componentwise shared pointwise map.
    Gates $g_\ell$, biases, and residual connections follow \eqref{eq:AMOR_block}; reference offsets $\mu$ follow Remark~\ref{rem:offsets}.}
    \label{fig:architecture}
\end{figure}

\textbf{Computational complexity.}
For $m$ retained Fourier modes per spatial direction, the FNO coefficient map costs
$\mathcal O(d_v^2m^d)$. The AMOR coefficient map costs
$\mathcal O(d_v^2k_{\mathrm{enc}}k_{\mathrm{dec}})$ in the dense case and
$\mathcal O(d_vk_{\mathrm{enc}}(d_v+k_{\mathrm{dec}}))$ in the factorized case.
Writing $k_{\mathrm{enc}}\asymp k_{\mathrm{dec}}\asymp k$, these become
$\mathcal O(d_v^2k^2)$ and $\mathcal O(d_v^2k+d_vk^2)$, respectively.
Thus neither general coefficient map is linear in $k$ at fixed $d_v$.
Only the dense basis transforms scale linearly in $k$ for fixed field size and width.
Although $d$ does not appear explicitly in the AMOR coefficient-map formulas,
the required $k$ and spatial resolution can depend on $d$.

For diagonal or uniformly sparse mass matrices, encoding and decoding cost
$\mathcal O(d_v(D_{\mathrm{enc}}k_{\mathrm{enc}}+
D_{\mathrm{dec}}k_{\mathrm{dec}}))$; a general sparse encoder also incurs
$\mathcal O(d_v\operatorname{nnz}(M_{\mathrm{enc}}))$ for mass multiplication.
The bases require $\mathcal O(D_{\mathrm{in}}k_{\mathrm{in}}+
D_{\mathrm{out}}k_{\mathrm{out}})$ storage and are shared across layers.
The target-mesh dual encoder can be precomputed; this preserves the same asymptotic
transform cost and storage order, with the offline cost stated in
Section~\ref{subsec:meshes}.
Except for learned bases, their construction is an offline cost; a full dense SVD
for POD costs $\mathcal O(\min(n_s,D)n_sD)$, with truncated or randomized alternatives
available when only leading modes are needed.

\begin{table}[htbp]
\centering
\small
\setlength{\tabcolsep}{4pt}
\renewcommand{\arraystretch}{1.25}
\begin{tabular}{lccc}
\toprule
Quantity & FNO & AMOR dense & AMOR factorized \\
\midrule
Transforms & $\mathcal O(d_vN\log N)$ & $\mathcal O(d_vDk)$ & $\mathcal O(d_vDk)$ \\
Coefficient map & $\mathcal O(d_v^2m^d)$ & $\mathcal O(d_v^2k_e k_d)$ & $\mathcal O(d_vk_e(d_v+k_d))$ \\
Standard pointwise map & $\mathcal O(d_v^2N)$ & $\mathcal O(Cd_v^2N)$ & $\mathcal O(Cd_v^2N)$ \\
Stored basis arrays & none & $\mathcal O(Dk)$ & $\mathcal O(Dk)$ \\
\bottomrule
\end{tabular}
\caption{Per-layer operation counts, with $k_e=k_{\mathrm{enc}}$,
$k_d=k_{\mathrm{dec}}$, $k=\max(k_e,k_d)$, and
$D=\max(D_{\mathrm{enc}},D_{\mathrm{dec}})$.
The coefficient-map expressions also give parameter counts.
The AMOR transition pointwise map instead costs
$\mathcal O(NC_{\mathrm{in}}C_{\mathrm{out}}d_v^2)$.
Here $d_v$ is the total FNO channel width and the AMOR width per component;
equal $d_v$ therefore does not mean equal total hidden width.
Projection, conditioning, activation, and skip costs are not included.}
\label{tab:complexity}
\end{table}

\subsection{Moving Between Meshes}\label{subsec:meshes}
We transport the reference basis without reorthonormalizing it. A mass-weighted
dual encoder accounts for the target-mesh geometry, so the learned modal
coordinates and coefficient tensor retain their reference meaning.

Let $\Phi_1\in\mathbb R^{D_1\times k}$ be a reference basis with
$\Phi_1^\top M_1\Phi_1=I_k$, and let $P=I_C\otimes P_x$ map reference fields
to the target mesh. Define the target decoder and its Gram matrix by
\begin{equation}\label{eq:mesh_transfer}
 B=P\Phi_1,\qquad G=B^\top M_2B.
\end{equation}
Assume $M_2$ is symmetric positive definite and $B$ has full column rank. The
target encoder is
\begin{equation}\label{eq:transfer_dual}
 E=G^{-1}B^\top M_2.
\end{equation}
Thus the decoder simply prolongates the reference modes, while the Gram correction
is included entirely in the encoder. In computation, $G$ is factored and the
associated linear systems are solved; an explicit matrix inverse is unnecessary.
The matrix $E$ may be precomputed once and shared across layers.

\begin{lemma}[Preservation of reference modal coordinates]\label{lem:dual_transfer}
The encoder and decoder in \eqref{eq:mesh_transfer}--\eqref{eq:transfer_dual} satisfy
\begin{equation}\label{eq:transfer_coordinates}
 EB=I_k,\qquad E(Bz)=z\quad(z\in\mathbb R^k).
\end{equation}
For arbitrary target data $h$, $Eh$ is the unique minimizer of
$\|h-Bz\|_{M_2}^2$, and $BE$ is the $M_2$-orthogonal projector onto
$\operatorname{range}(B)$.
\end{lemma}
\begin{proof}
The first identity is $EB=G^{-1}B^\top M_2B=I_k$.
The normal equations for the strictly convex least-squares problem are
$Gz=B^\top M_2h$, whose solution is $z=Eh$.
The residual therefore satisfies $B^\top M_2(h-BEh)=0$, proving the projector claim.
\end{proof}

The distinction from the reference-mesh formulation is explicit:
$B^\top M_2$ is the mass adjoint of the decoder, whereas $E$ is its dual
coefficient extractor. They coincide when $G=I_k$. In particular, on the reference
mesh $B=\Phi_1$ and $E=\Phi_1^\top M_1$.

Apply this construction separately to the encoding and decoding bases of each
block. With the original tensor $T$ unchanged, the target block is
\begin{equation}\label{eq:transferred_block}
 \mathcal K_2H=B_{\mathrm{dec}}\bigl[T(E_{\mathrm{enc}}H)\bigr].
\end{equation}
For any coefficient array $Z$, Lemma~\ref{lem:dual_transfer} gives
\[
 \mathcal K_2(B_{\mathrm{enc}}Z)
   =B_{\mathrm{dec}}(TZ)
   =P_{\mathrm{dec}}\Phi_{1,\mathrm{dec}}(TZ).
\]
This is exactly the prolongation of the reference block output for input
$\Phi_{1,\mathrm{enc}}Z$. No change to $T$ is required, and its factorized form
is preserved. Offsets are transported by $\mu_2=P\mu_1$ and are subtracted and
added columnwise as in Remark~\ref{rem:offsets}.

These identities preserve coefficient coordinates exactly on the transported
modal spaces. They do not assert equality of complete network evaluations across
meshes: hidden states can leave those spaces, and interpolation generally does
not commute with pointwise nonlinearities. For arbitrary reference fields, exact
agreement of the encodings additionally requires $E_2P=E_1$. A sufficient
condition is $P^\top M_2P=M_1$, as for exact inclusion of nested finite-element
spaces with exact mass matrices. General transfer accuracy requires consistency
of the field discretizations and transport, and is assessed separately from the
continuum universality theorem. See~\ref{subsec:transfer} reports the mesh-transfer experiments.

The full-column-rank assumption excludes transfers that collapse distinct modes.
The conditioning of $G$ controls the sensitivity of the dual-coordinate solve.
For one basis, forming $B$, $G$, and a precomputed $E$ costs
$\mathcal O(k\operatorname{nnz}(P)+k\operatorname{nnz}(M_2)+D_2k^2+k^3)$.
Storing both $B$ and $E$ requires $\mathcal O(D_2k)$ memory, and each subsequent encoding or decoding costs $\mathcal O(D_2kd_v)$ for $d_v$ channels.

\subsection{Reduced Basis Constructions}\label{subsec:bases_taxonomy}
We use the discrete notation above, suppressing input/output superscripts when
unambiguous. Each construction produces an orthonormal basis on the reference
mesh, $\Phi\in\mathbb R^{D\times k}$ with $\Phi^\top M\Phi=I_k$, subject to
its stated rank assumptions. The transport in Section~\ref{subsec:meshes} then
retains these reference modes without target-mesh reorthonormalization.
For a full-column-rank matrix $\Xi$, symmetric (L\"owdin) orthonormalization is
\begin{equation}\label{eq:lowdin}
    \mathrm{orth}_M(\Xi)=\Xi G_\Xi^{-1/2},\qquad G_\Xi=\Xi^\top M\Xi.
\end{equation}
It is the unique minimizer of $\|M^{1/2}(\Phi-\Xi)\|_F$ over
$\Phi^\top M\Phi=I_k$~\cite{lowdin1950nonorthogonality}, since $M^{1/2}\Phi$ is the orthonormal polar factor of $M^{1/2}\Xi$~\cite[Ch.~8]{higham2008functions}.
An alternative is unpivoted $M$-weighted QR, which preserves the leading column spans.
For $M=\Delta V I_D$, Euclidean orthonormal directions are scaled by
$(\Delta V)^{-1/2}$ to obtain $M$-orthonormal modes.

POD and DIS produce vector-valued modes directly. Harmonic and polynomial modes are first constructed for scalar fields and then lifted as
$\Phi=I_C\otimes\Phi_{\mathrm{sc}}$.
Before any optional constant-mode augmentation, this gives $k=Ck_{\mathrm{sc}}$;
reported mode counts include all components and any augmentation.
Theorem~\ref{thm:AMOR_universality} requires the normalized global constant, added if necessary by Remark~\ref{rem:constant_mode}. 
Discrete orthonormality alone does not establish continuum continuity or point separation.

\textbf{POD/PCA~\cite{sirovich1987turbulence,berkooz1993proper}.}
For snapshots arranged rowwise in $Y\in\mathbb R^{n_s\times D}$, let $\underline\mu$
be the sample mean, $Y_0=Y-\mathbf1\underline\mu^\top$, and
$\Sigma=n_s^{-1}Y_0^\top Y_0$.
Maximizing mean squared projected energy subject to $M$-normalization gives
\begin{equation}\label{eq:pca}
    \Sigma M\Phi=\Phi\operatorname{diag}(\lambda_1,\ldots,\lambda_k),
    \qquad \lambda_1\ge\cdots\ge\lambda_k,
    \qquad \Phi=M^{-1/2}V_k,
\end{equation}
where $V_k$ contains leading right singular vectors of $Y_0M^{1/2}$.
The matrix $\Sigma M$ is self-adjoint in the $M$ inner product.
For positive-variance modes, $k$ cannot exceed $\operatorname{rank}(Y_0)$.
The reference offsets are the input and output sample means.
Centering snapshots does not imply that the modes are spatially mean-zero.

\textbf{Harmonic (Laplacian)~\cite{belkin2003laplacian,levy2006laplace}.}
The scalar modes solve
\begin{equation}\label{eq:laplacian}
    L_\Delta\Phi_{\mathrm{sc}}=M_x\Phi_{\mathrm{sc}}\operatorname{diag}(\lambda_1,\ldots,\lambda_{k_{\mathrm{sc}}}),
    \qquad \lambda_1\le\cdots\le\lambda_{k_{\mathrm{sc}}},
    \qquad \Phi_{\mathrm{sc}}^\top M_x\Phi_{\mathrm{sc}}=I_{k_{\mathrm{sc}}}.
\end{equation}
Here $L_\Delta$ is a finite-element stiffness matrix with the specified boundary
conditions, or a symmetric graph Laplacian
$L_\Delta=\operatorname{diag}(\mathcal A\mathbf1)-\mathcal A$ with nonnegative
symmetric adjacency $\mathcal A$.
Boundary conditions and graph construction are part of the basis specification.
Neumann and standard graph Laplacians contain constant null modes; Dirichlet
eigenfunctions generally do not. Repeated eigenvalues leave a choice of basis
within the eigenspace, so bases fitted independently on different meshes require
alignment before reusing mode-indexed weights.
This construction depends on the chosen geometry and differential or graph operator,
not on the training snapshots.

\textbf{Learned orthonormal bases~\cite{helfrich2018orthogonal}.}
A trainable matrix $\Xi\in\mathbb R^{D\times k}$ is orthonormalized in each forward
pass by \eqref{eq:lowdin} or $M$-weighted QR. Full column rank is required; near rank
deficiency also affects numerical conditioning. A Cayley parametrization supplies
an optional rotation of the modal coordinates:
\begin{equation}\label{eq:learned_cayley}
    \Gamma=(I_k+S)^{-1}(I_k-S),\qquad
    S^\top=-S,\qquad \widetilde\Phi=\Phi\Gamma^\top.
\end{equation}
Since $I_k+S$ is invertible and $\Gamma^\top\Gamma=I_k$, orthonormality is preserved.
The skew parameters are unrestricted, but the Cayley parametrization does not
cover orthogonal matrices with eigenvalue $-1$.
For an unrestricted dense coefficient tensor, a modal rotation can be absorbed
into that tensor; it changes the parametrization rather than enlarging the
represented reduced space. These are the bases optimized jointly with the network.

\textbf{Polynomial bases~\cite{szego1939orthogonal}.}
Rescale each coordinate affinely to $[-1,1]$ and evaluate tensor-product Jacobi
polynomials, ordered by total degree:
\begin{equation}\label{eq:polynomial}
    \widetilde\phi_{\mathbf n}(x)
      =\prod_{i=1}^d\mathcal P_{n_i}^{(\alpha,\beta)}(x_i),
    \qquad |\mathbf n|\le n_{\max}.
\end{equation}
The choices $(\alpha,\beta)=(1,1),(0,0),(-\tfrac12,-\tfrac12)$ give a Jacobi family,
Legendre polynomials, and polynomials proportional to the Chebyshev family,
respectively. Their classical orthogonality concerns the corresponding weighted
product measure on the full reference box; it need not hold on a general domain
or under the discrete mass inner product.
Provided the sampled polynomial matrix has full column rank, set
$\Phi_{\mathrm{sc}}=\mathrm{orth}_{M_x}(\widetilde\Phi)$ and lift it componentwise.
With total degree at least one, the underlying continuous scalar space contains the constants and coordinates and therefore separates spatial points.
After componentwise lifting, the corresponding continuous polynomial spans contain the constant function and coordinates, and thereby satisfy the point-separation requirement of Theorem~\ref{thm:AMOR_universality}; the numerical implementations are orthonormalized in the chosen numerical inner product. 

\textbf{Derivative-informed subspaces (DIS)~\cite{zahm2019gradientbaseddimensionreductionmultivariate,o2022derivative,olearyroseberry2024derivative}.}
DIS assumes a Gaussian input law $\mathcal N(\underline\mu_{\mathrm{in}},\mathcal C_{\mathrm{pr}})$ and uses Jacobians $J=\partial\Psi_h^\dagger/\partial\underline a$ of the discrete forward map at sampled inputs. With
\[
 \Sigma_{\mathrm{in}}=\mathbb E[J^\top M_{\mathrm{out}}J],\qquad
 \Sigma_{\mathrm{out}}=\mathbb E[J\,\mathcal C_{\mathrm{pr}}J^\top],
\]
the input and output modes solve
\begin{equation}\label{eq:dis}
 \Sigma_{\mathrm{in}}\Phi_{\mathrm{in}}=\mathcal C_{\mathrm{pr}}^{-1}\Phi_{\mathrm{in}}\Lambda_{\mathrm{in}},
 \qquad
 \Sigma_{\mathrm{out}}M_{\mathrm{out}}\Phi_{\mathrm{out}}=\Phi_{\mathrm{out}}\Lambda_{\mathrm{out}},
\end{equation}
with $\Phi_{\mathrm{in}}^\top\mathcal C_{\mathrm{pr}}^{-1}\Phi_{\mathrm{in}}=I$, $\Phi_{\mathrm{out}}^\top M_{\mathrm{out}}\Phi_{\mathrm{out}}=I$, and eigenvalues ordered decreasingly. In this prior metric, the discarded input eigenvalues bound the mean-squared error of the best approximation that depends only on the retained input coordinates~\cite{zahm2019gradientbaseddimensionreductionmultivariate}. The output modes are the $M_{\mathrm{out}}$-weighted POD of Jacobian actions $J\underline z_i$ on prior samples $\underline z_i\sim\mathcal N(0,\mathcal C_{\mathrm{pr}})$. The input encoder is the $\mathcal C_{\mathrm{pr}}^{-1}$-adjoint of the decoder, $E_{\mathrm{in}}=\Phi_{\mathrm{in}}^\top\mathcal C_{\mathrm{pr}}^{-1}=\Lambda_{\mathrm{in}}^{-1}\Phi_{\mathrm{in}}^\top\Sigma_{\mathrm{in}}$, so that $E_{\mathrm{in}}\Phi_{\mathrm{in}}=I$ without whitening the input; the output encoder is the $M_{\mathrm{out}}$-adjoint. All required actions are computed from Jacobian--vector and adjoint products, without forming $J$; we represent $\mathcal C_{\mathrm{pr}}$ by a low-rank approximation plus a small diagonal shift. With a finite-element discretization of $\mathcal C_{\mathrm{pr}}$, the construction applies unchanged to unstructured meshes. DIS uses derivatives of the target map, whereas POD uses snapshots and the harmonic construction a spatial operator.

\section{Methods}\label{sec:methods}

We evaluate AMOR on eight benchmark problems spanning elliptic, hyperbolic, parabolic, and solid-mechanics PDEs on both structured and unstructured meshes.

\subsection{Data}\label{subsec:data}
Table~\ref{tab:datasets} summarizes the eight datasets; solver, element type, boundary conditions, and parameter ranges are given in full in ~\ref{si:data}.

\begin{table}[htbp]
\centering\small
\caption{Benchmark datasets. Train/valid/test split $70/20/10$, except Airfoil and Elasticity (Geo-FNO split $1000/200$).}
\label{tab:datasets}
\begin{tabular}{llllr}
\toprule
Dataset & Operator map & Mesh & Resolution/nodes & $n_s$ \\
\midrule
Nonlinear diffusion & $(f,m)\mapsto u$ & regular & $100\times100$ & 10,000 \\
Helmholtz & $m\mapsto(u_r,u_i)$ & regular (P1) & $100\times100$ & 10,000 \\
Kolmogorov flow~\cite{li2022mno, li2022mno_zenodo} & $u(t)\mapsto u(t+\Delta t)$ & regular & $64^2/128^2$ & 200 traj. \\
Airfoil~\cite{li2023fourier} & geometry $\mapsto$ Mach number & curvilinear & $221\times51$ & 1,200 \\
Elasticity~\cite{li2023fourier} & geometry $\mapsto$ stress & point cloud & 972 nodes & 2,000 \\
Elastostatics 3D & $E\mapsto\mathbf u$ & regular ($Q_1$) & $64^3$ & 1,000 \\
FlowBench LDC~\cite{tali2024flowBench} & geometry $\mapsto(\mathbf u,p)$ & regular & $128^3$ & 1,000 \\
Aircraft rib & $(E,\mathbf t_\Gamma)\mapsto\mathbf u$ & irregular (P1) & 72,615 nodes & 10k--40k \\
\bottomrule
\end{tabular}
\end{table}

\emph{Nonlinear diffusion} solves
\begin{equation}
    -\nabla\cdot(e^{m}\nabla u) = f \quad \text{on } [0,1]^2,
\end{equation}
with homogeneous Dirichlet conditions and $f,m$ independent Gaussian-process draws at three correlation lengths $(\ell_m,\ell_f)$, giving input--output pairs $((f,m),u)$.

\emph{Helmholtz} solves the fixed-frequency ($600\,\mathrm{Hz}$) system with wavenumber $k_{\mathrm w}(\mathbf x)=k_{\mathrm{ref}}\exp(m(\mathbf x))$ and PML absorption on $[0,3]^2$, with $m\sim\textsf{BiLaplacian}(\sigma,\rho)$ (three variants), yielding pairs $(m,(u_r,u_i))$.

\emph{Kolmogorov flow}~\cite{li2022mno, li2022mno_zenodo} solves the 2D Navier--Stokes equations at $\mathrm{Re}\in\{40,500,5000\}$; the datasets differ in spatial resolution ($64^2$ vs.\ $128^2$) and in the ratio of the snapshot interval to the flow decorrelation time.

\emph{Linear elastostatics (3D)} solves
\begin{equation}
    -\nabla\cdot\boldsymbol\sigma(\mathbf u) = \mathbf 0, \qquad
    \boldsymbol\sigma = \lambda\,\mathrm{tr}(\boldsymbol\varepsilon(\mathbf u))\mathbf I + 2\mu\,\boldsymbol\varepsilon(\mathbf u),
\end{equation}
on a clamped, uniformly loaded unit cube, with Lam\'e fields derived from a log-normal Young's modulus field $E(\mathbf x)$ at fixed Poisson ratio; input-output pairs are $(E,\mathbf u)$.

\emph{FlowBench LDC} is the public 3D lid-driven-cavity Navier--Stokes dataset of~\cite{tali2024flowBench}, with varying internal geometry (ellipsoids, toroids, boxes, cylinders).

\emph{Airfoil} and \emph{Elasticity}~\cite{li2023fourier} are the Geo-FNO benchmarks on a curvilinear structured mesh (Airfoil) and a sample-dependent point cloud (Elasticity), predicting the Mach-number field and an internal stress field, respectively, on geometry-varying meshes.

\emph{Aircraft rib.} All three aircraft-rib datasets solve the same static linear-elasticity problem
\begin{equation}
    -\nabla\cdot\boldsymbol\sigma(\mathbf u) = \mathbf 0 \ \text{in } \Omega, \qquad
    \boldsymbol\sigma = 2\mu\,\boldsymbol\varepsilon(\mathbf u) + \lambda\,\mathrm{tr}(\boldsymbol\varepsilon(\mathbf u))\mathbf I,
\end{equation}
on a single tetrahedral rib mesh ($n_{\mathrm{node}}=72{,}615$), clamped at the root, with $E(\mathbf x)$ a log-normal random field. The three versions differ in where and how the Neumann traction is applied and, for the multipatch bracket and distributed aero versions, in a variable Poisson ratio (Fig.~\ref{fig:aircraft_rib_dataset}), with loading configurations motivated by, but not trying to reproduce, structural test scenarios reported in the aerospace literature~\cite{Petrasinovic2017,IJRSIWingSpar,Esposito2023,RIBES,CARNEIRO2019708,ANNRibLoad,IJERTBracket}:
\begin{itemize}[nosep]
    \item \emph{Cantilever tip} (10,000 samples): axial traction on a fixed tip patch.
    \item \emph{Distributed aero} (10,000 samples): uniform traction over all admissible facets away from the clamp.
    \item \emph{Multipatch bracket} (20,000 samples): 1--3 randomly placed, independently oriented local traction patches.
\end{itemize}
A \emph{combined} dataset pools all three ($40{,}000$ samples, load included as an input field). In every sample the total load is rescaled to keep the maximum deflection within a small-displacement regime ($0.01$--$0.3\%$ of rib length).
\begin{figure}[htbp]
    \centering
    \includegraphics[width=\linewidth]{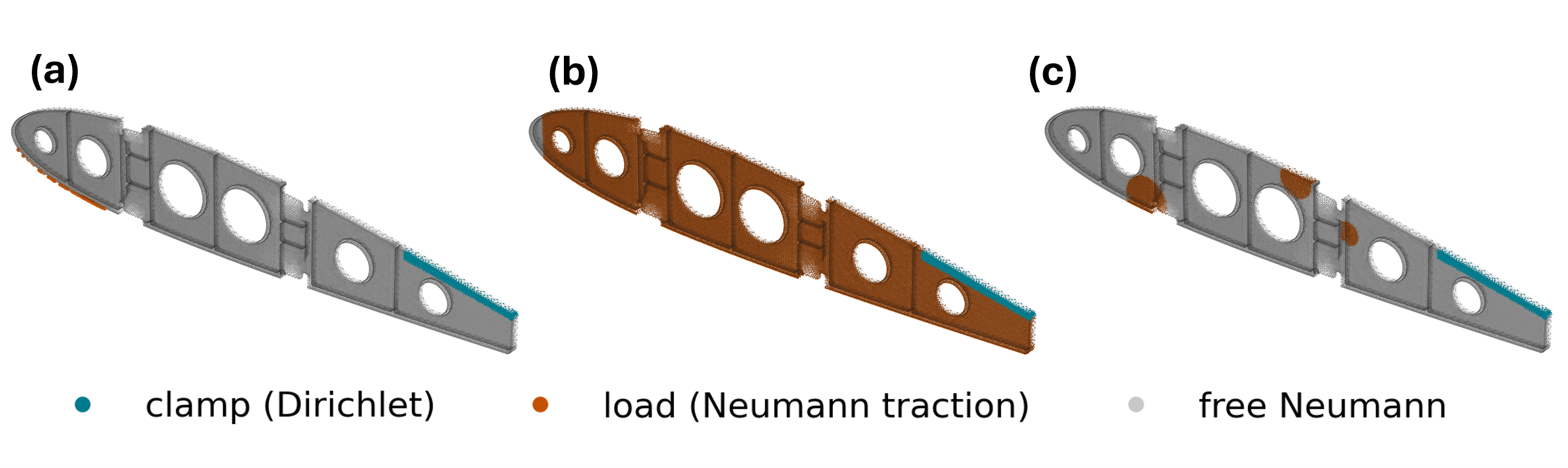}
    \caption{Boundary conditions of the three aircraft-rib datasets. 
    Dirichlet clamp depicted in cyan (constant over all versions), Neumann load in orange and traction-free in grey. Applied load on a fixed tip-patch with fixed direction in \emph{cantilever tip} (a), uniformly excluding tip and clamping in \emph{distributed aero} (b) and on $1$--$3$ localized patches with independent load directions \emph{multipatch bracket} (c).}
    \label{fig:aircraft_rib_dataset}
\end{figure}

\subsection{Model configurations and training}\label{sec:model_config}

We compare AMOR with FNO~\cite{li2021fourier} and PCANet~\cite{bhattacharya2021model}. 
AMOR and FNO use the same channel width and the same number of nonlocal layers, but parametrize their nonlocal operators differently (\ref{si:params}); for AMOR, the nonlocal layers are the input, transition, and residual blocks. 
On the two-dimensional benchmarks, FNO retains $12\times12$ Fourier modes, and AMOR and PCANet use $k_{\mathrm{in}}=k_{\mathrm{out}}=144$ modes of the same PCA basis; PCANet uses its own width and depth. Elsewhere, AMOR retains $99.9\%$ of the training variance, except for FlowBench (Laplacian basis) and the aircraft rib, whose ranks are selected on validation data. Coefficient tensors are factorized whenever $k_{\mathrm{enc}}k_{\mathrm{dec}}>10^4$. Parameter counts are listed in  Table~\ref{tab:model_configs}.

All models are trained with AdamW~\cite{kingma2015adam,loshchilov2019decoupledweightdecayregularization} (learning rate $2\times10^{-3}$, weight decay $10^{-4}$, batch size fixed per dataset (details in~\ref{si:computational_details}), cosine annealing, at most $300$ epochs), gradient clipping at norm $3$, and early stopping on the validation relative $L^1$ error with patience $50$; the checkpoint with the lowest validation error is evaluated. The loss is the relative $L^1$ error, except on the Geo-FNO benchmarks, which follow~\cite{li2023fourier}. Data are split $70/20/10$ into training, validation, and test sets (by geometry for FlowBench), except for the Geo-FNO benchmarks (published $1000/200$ split). Bases and channel-wise normalization statistics are fit on the training split only. We report relative $L^1$ and $L^2$ test errors of the denormalized fields.
Except for the aircraft rib, every model is trained with five seeds that vary only the initial weights; split, normalization, batch order, and fixed bases are identical across seeds.
We report mean and standard deviation; the rib results use a single seed. Hardware is described in~\ref{si:hardware}.

\section{Results and Discussion}\label{sec:results}

\subsection{Structured grid benchmarks in two dimensions}\label{subsec:benchmarks}

Section~\ref{subsec:compet_FNO} compares AMOR with FNO and PCANet on structured grids, where FNO applies. AMOR and PCANet use the same PCA basis, which isolates the effect of the architecture; Section~\ref{subsec:basis} varies the basis.
We then turn to the irregular-mesh and three-dimensional problems that motivate the AMOR architecture (Sec.~\ref{subsec:irregular}, Sec.~\ref{subsec:3d}).
Exemplary predictions are shown in~\ref{si:visualization}.

\subsubsection{AMOR is competitive with FNO on structured grids}\label{subsec:compet_FNO}

Table~\ref{tab:benchmark} summarises the performance of AMOR, FNO, and PCANet on the three structured-grid benchmarks where FNO is applicable.
\begin{table}[htbp]
\centering
\caption{Performance comparison (relative errors [\%], mean $\pm$ standard deviation over five seeds) of AMOR, FNO, and PCANet on benchmark problems. Best results in \textbf{bold}.  AMOR is more accurate than FNO on the nonlinear-diffusion and Helmholtz problems, and substantially outperforms PCANet despite using the same basis; the exception is the transport-dominated Kolmogorov-flow problem, where FNO performs better.}
\label{tab:benchmark}
\small
\begin{tabular}{llcccc}
\toprule
\textbf{Dataset} & \textbf{Model} & \textbf{$L^1$ Error} & \textbf{$L^2$ Error} & \textbf{Infer. Time} & \textbf{Train Time} \\
&  & [\%] & [\%] & [ms/sample] & [s/epoch] \\
\midrule
\midrule
Nonlinear & AMOR         & $\mathbf{0.50\pm0.02}$ & $\mathbf{0.54\pm0.02}$ & 0.44 & 6.10 \\
Diffusion & FNO          & $1.75\pm0.10$ & $1.73\pm0.10$ & 0.23 & 3.58 \\
$(\ell_m, \ell_f)=(0.5, 0.4)$ & PCANet       & $14.26\pm0.54$ & $14.36\pm0.52$ & \textbf{0.02} & \textbf{0.37} \\
\midrule
Nonlinear & AMOR         & $\mathbf{1.97\pm0.09}$ & $\mathbf{2.07\pm0.09}$ & 0.44 & 6.14 \\
Diffusion & FNO          & $4.27\pm0.44$ & $4.15\pm0.42$ & 0.23 & 3.55 \\
$(\ell_m, \ell_f)=(0.3, 0.2)$ & PCANet       & $31.35\pm0.76$ & $31.90\pm0.79$ & \textbf{0.02} & \textbf{0.12} \\
\midrule
Nonlinear & AMOR         & $\mathbf{10.87\pm0.39}$ & $\mathbf{10.97\pm0.37}$ & 0.44 & 6.12 \\
Diffusion & FNO          & $12.03\pm0.85$ & $11.16\pm0.77$ & 0.23 & 3.22\\
$(\ell_m, \ell_f)=(0.15, 0.1)$ & PCANet       & $57.05\pm0.64$ & $60.42\pm0.96$ & \textbf{0.02} & \textbf{0.21} \\
\midrule
\midrule
Helmholtz & AMOR       & $\mathbf{1.70\pm0.02}$ & $\mathbf{1.70\pm0.02}$ & 0.312 & 6.49 \\
& FNO         & $4.47\pm0.07$ & $5.09\pm0.07$ & 0.244 & 9.38 \\
$(\sigma, \rho)=(0.1, 0.6)$ & PCANet      & $3.24\pm0.05$ & $3.25\pm0.05$ & \textbf{0.090} & \textbf{3.66} \\
\midrule
Helmholtz & AMOR       & $\mathbf{5.18\pm0.05}$ & $\mathbf{5.12\pm0.04}$ & 0.296 & 6.49 \\
& FNO         & $10.61\pm0.43$ & $11.35\pm0.38$ & 0.235 & 9.30 \\
$(\sigma, \rho)=(0.2, 0.4)$ & PCANet      & $8.46\pm0.12$ & $8.43\pm0.12$ & \textbf{0.094} & \textbf{3.66} \\
\midrule
Helmholtz & AMOR       & $\mathbf{18.71\pm0.13}$ & $\mathbf{18.45\pm0.12}$ & 0.395 & 9.92 \\
& FNO         & $27.12\pm1.39$ & $27.87\pm1.54$ & 0.292 & 9.86 \\
$(\sigma, \rho)=(0.4, 0.2)$ & PCANet      & $26.03\pm0.48$ & $25.91\pm0.46$ & \textbf{0.091} & \textbf{3.94} \\
\midrule
\midrule
Kolmogorov & AMOR       & $7.88\pm0.19$  & $7.96\pm0.20$ & 0.092 & 27.17 \\
Flow & FNO         & $\mathbf{3.48\pm0.04}$ & $\mathbf{3.54\pm0.04}$ & 0.077 & 23.30 \\
$\mathrm{Re}=40$ & PCANet      & $35.06\pm0.07$ & $37.81\pm0.03$ & \textbf{0.029} & \textbf{7.48} \\
\midrule
Kolmogorov & AMOR       & $36.59\pm0.34$ & $32.60\pm0.33$ & 0.149 & 243.36 \\
Flow & FNO         & $\mathbf{28.90\pm0.20}$ & $\mathbf{26.07\pm0.17}$ & 0.111 & 164.54 \\
$\mathrm{Re}=500$ & PCANet      & $89.29\pm0.86$ & $86.28\pm1.72$ & \textbf{0.029} & \textbf{36.94} \\
\midrule
Kolmogorov & AMOR       & $50.29\pm0.37$ & $45.97\pm0.31$ & 0.731 & 102.29 \\
Flow & FNO         & $\mathbf{29.60\pm0.95}$ & $\mathbf{27.79\pm0.95}$ & 0.479 & 66.27 \\
$\mathrm{Re}=5000$ & PCANet      & $61.59\pm0.01$ & $59.02\pm0.01$ & \textbf{0.051} & \textbf{5.32} \\
\bottomrule
\end{tabular}
\end{table}
AMOR is more accurate than FNO on all nonlinear-diffusion and Helmholtz variants, with the smallest margin on the roughest diffusion variant ($\ell_m=0.15$), and more accurate than PCANet on every problem. Since AMOR and PCANet share the basis and the number of modes, the latter gap reflects AMOR's repeated nonlocal blocks and physical-space nonlinearities. AMOR's inference and training times are within a factor of about two of FNO's.

FNO is more accurate on all three Kolmogorov-flow variants. This periodic, statistically homogeneous flow matches FNO's translation equivariance (Remark~\ref{rem:equivariance}), and its one-step increments are poorly represented by a static snapshot basis. Learned and Laplacian bases reduce, but do not close, the gap (Table~\ref{tab:basis}).

\subsubsection{Choice of Basis}\label{subsec:basis}
The benchmarks above fix a PCA basis for fairness of comparison. 
AMOR is, however, agnostic to the choice of reduced basis: only the encode and decode maps depend on it, while the deep cross-basis architecture is unchanged. 
Table~\ref{tab:basis} replaces the PCA basis by a Laplacian eigenbasis, orthogonal polynomials, learned orthonormal bases, and derivative-informed subspaces (DIS)~\cite{olearyroseberry2024derivative} (omitted for Kolmogorov flow as the dataset generation code was not accessible to the authors), with all other hyperparameters fixed; only the encoder and decoder change.
\begin{table}[htbp]
\centering
\caption{AMOR basis comparison on three structured grid benchmark datasets. Best results in \textbf{bold}. Learned bases give the lowest errors throughout but increase training cost. The best performing precomputed basis is problem-dependent: On Nonlinear Diffusion $(\ell_m,\ell_f)=(0.3,0.2)$ polynomial bases work best, on Helmholtz $(\sigma,\rho)=(0.2,0.4)$ the DIS basis, and on Kolmogorov flow $\mathrm{Re}=40$ the Laplacian basis.}
\label{tab:basis}
\begin{tabular}{lcccccc}
\toprule
& \multicolumn{2}{c}{Nonlinear Diffusion} & \multicolumn{2}{c}{Helmholtz} & \multicolumn{2}{c}{Kolmogorov Re40} \\
\cmidrule(lr){2-3} \cmidrule(lr){4-5} \cmidrule(lr){6-7}
Basis & $L^1$ [\%] & $L^2$ [\%] & $L^1$ [\%] & $L^2$ [\%] & $L^1$ [\%] & $L^2$ [\%] \\
\midrule
PCA                         & $1.97\pm0.09$ & $2.07\pm0.09$ & $5.18\pm0.05$ & $5.12\pm0.04$ & $7.88\pm0.19$ & $7.96\pm0.20$ \\
Laplacian                   & $1.64\pm0.09$ & $1.72\pm0.08$ & $5.06\pm0.05$ & $5.04\pm0.06$ & $5.73\pm0.20$ & $5.85\pm0.21$ \\
Learned Cayley              & $1.39\pm0.04$ & $1.48\pm0.05$ & $3.21\pm0.05$ & $3.17\pm0.05$ & $5.16\pm0.31$ & $5.32\pm0.32$ \\
Learned QR                  & $\mathbf{1.35\pm0.04}$ & $\mathbf{1.44\pm0.05}$ & $\mathbf{3.09\pm0.03}$ & $\mathbf{3.05\pm0.03}$ & $\mathbf{4.89\pm0.18}$ & $\mathbf{5.01\pm0.18}$ \\
Jacobi                      & $1.64\pm0.05$ & $1.78\pm0.05$ & $4.97\pm0.02$ & $4.95\pm0.02$ & $7.33\pm0.18$ & $7.46\pm0.19$ \\
Chebyshev                   & $1.51\pm0.05$ & $1.63\pm0.05$ & $4.99\pm0.05$ & $4.96\pm0.06$ & $7.00\pm0.08$ & $7.12\pm0.09$ \\
Legendre                    & $1.52\pm0.06$ & $1.65\pm0.06$ & $4.99\pm0.05$ & $4.96\pm0.04$ & $6.97\pm0.09$ & $7.09\pm0.10$ \\
DIS                         & $1.72\pm0.08$ & $1.84\pm0.08$ & $3.31\pm0.05$ & $3.29\pm0.05$ & -- & -- \\
\bottomrule
\end{tabular}
\end{table}
The learned bases give the lowest errors on all three problems and reduce the PCA error by $19$--$42\%$, but they add trainable parameters and an orthonormalization in every training step, which increases the training time per epoch by a factor of $3.5$--$6.4$ in our runs. 
The basis is a modular, problem-matched component. Among the precomputed bases, which leave the cost unchanged, the best choice is problem-dependent: the Laplacian basis improves on PCA on all three problems, DIS nearly matches the learned bases on Helmholtz and on the nonlinear diffusion problem polynomial bases result in the best accuracy.

\subsubsection{Scaling to three dimensional problems}\label{subsec:3d}
We evaluate FNO and AMOR on two three-dimensional operator-learning benchmarks, a lid-driven cavity from the public FlowBench~\cite{tali2024flowBench} database and a linear Elastostatics problem as shown in Table~\ref{tab:scaling}.
\begin{table}[htbp]
\centering
\caption{Operator-learning performance of FNO and AMOR on the 3D FlowBench and Elastostatics benchmarks. Best results in \textbf{bold}. AMOR is competitive with FNO on both problems and is simultaneously more accurate and faster on Elastostatics, while on FlowBench, FNO is more accurate and faster.}
\label{tab:scaling}
\begin{tabular}{l  l  c c c c}
\toprule
Dataset & Model & $L^1$ & $L^2$ & Infer. & Train time \\
        &       &      [\%]       &      [\%]        & [ms/sample] & [s/epoch] \\
\midrule
\midrule
\multirow{2}{*}{FlowBench}
  & FNO  & $\mathbf{17.58\pm0.86}$ & $\mathbf{21.16\pm0.99}$ & \textbf{47} & \textbf{121} \\
  & AMOR & 18.60$\pm$0.75 & 22.33$\pm$0.75 & 133 & 245 \\
\midrule
\multirow{2}{*}{Elastostatics}
  & FNO & 2.41$\pm$0.02 & 2.31$\pm$0.03 & 6.1 & 28.3 \\
  & AMOR  & $\mathbf{1.37\pm0.03}$ & $\mathbf{1.53\pm0.04}$ & \textbf{5.3} & \textbf{17.2} \\
  \bottomrule
\end{tabular}
\end{table}
On the Elastostatics dataset we retain a PCA basis with modes set by retaining $99.9\%$ of explained variance. 
On FlowBench, by contrast, we use a Laplacian eigenbasis: the varying internal geometry (ellipsoids, toroids, boxes, cylinders) masks out a different interior region per sample, so the leading PCA mode alone already saturates any practical variance threshold without resolving the flow field whereas the Laplacian eigenbasis, fit to the mesh rather than to the data, is unaffected.

To probe this further, we measure inference and training time on the Elastostatics dataset as a function of the two hyperparameters that principally control model capacity -- the number of retained Fourier modes (FNO) or basis functions (AMOR), and the channel width -- since AMOR's coefficient map has no tensor-product dependence on the spatial dimension whereas FNO's spectral weight tensor grows as $m^d$ (Section~\ref{subsec:amor_architecture}, Table~\ref{tab:complexity}).
\begin{figure}[htbp]
    \centering
    \includegraphics[width=\linewidth]{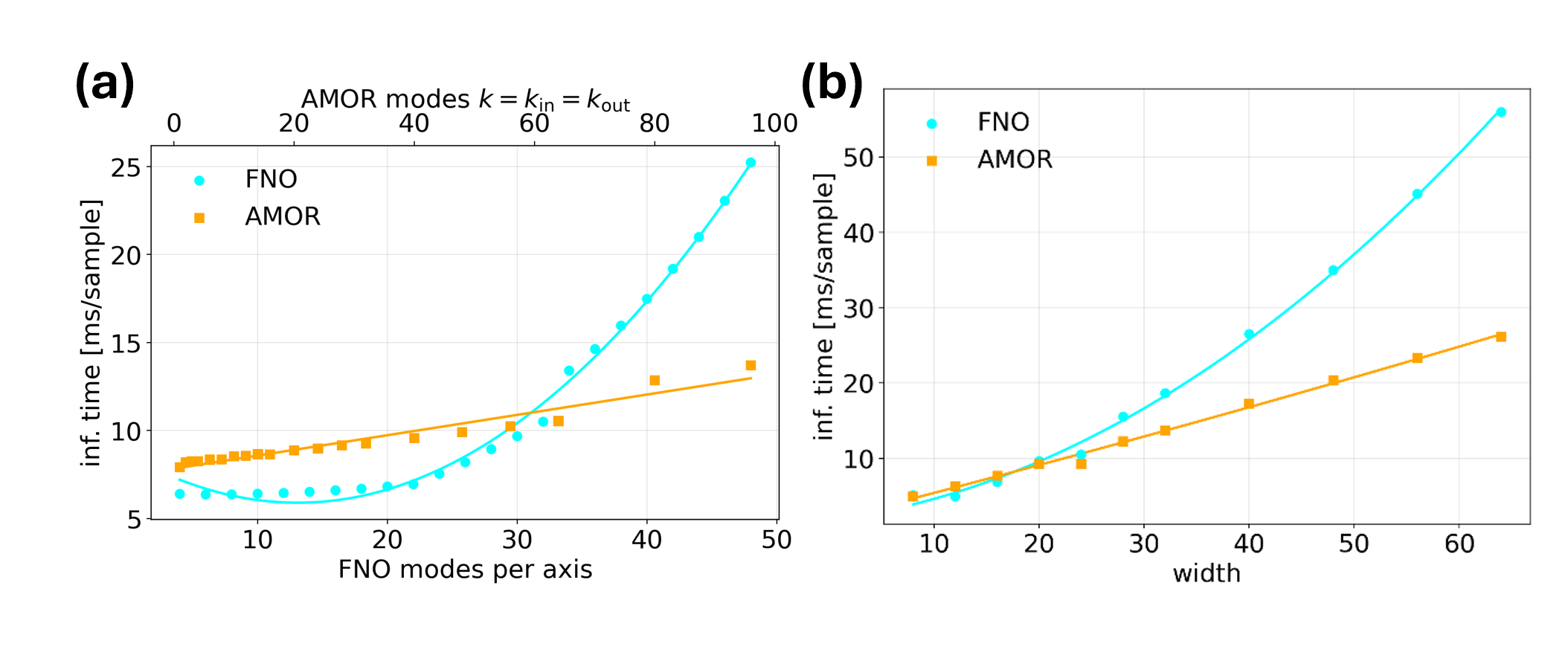}
    \caption{Inference time scaling of AMOR and FNO on the Elastostatics dataset. Scaling with respect to the number of modes per axis for $m = m_x = m_y = m_z$ for FNO and the number of input and output basis functions for $k = k_\mathrm{in} = k_\mathrm{out}$ for AMOR with all other hyperparameters and compute resources held fixed (a). Scaling with respect to the channel width for both AMOR and FNO with all other hyperparameters and compute resources held fixed (b). FNO scales cubically/quadratically with respect to modes/width, while AMOR's time grows approximately linearly/quadratically over the tested range (polynomial fits) shown as cyan/orange lines for FNO/AMOR. AMOR's cost grows markedly more slowly along both axes, so over the tested range its computational advantage over FNO widens toward higher modes and widths.}
    \label{fig:scaling}
\end{figure}
Figure~\ref{fig:scaling} (a) sweeps the number of modes at a fixed width of $24$.
FNO's inference time is flat at about $6.4$--$6.9$ ms/sample up to $m\approx22$ and then grows sharply, reaching $25.2$ ms/sample at $m=48$, consistent with the $m^3$ growth of its spectral weight tensor in three dimensions.
AMOR's time grows approximately linearly over the tested range, consistent with the $\mathcal O(d_vDk)$ transforms dominating the coefficient map.
Figure~\ref{fig:scaling} (b) sweeps channel width at a fixed mode/basis count of $32$ per axis for FNO/both input and output for AMOR.
Both costs grow with width, as expected from the quadratic channel-mixing term common to both architectures, but FNO grows markedly steeper: from $5.1$ to $55.9$ ms/sample between width $8$ and $64$ ($11\times$), against $5.0$ to $26.1$ ms/sample for AMOR ($5.3\times$).
Training-time scaling with modes and width follows the same qualitative pattern in a more pronounced way, because each optimization step adds a backward pass through the same $m^d$ spectral (FNO) or $k{\times}k$ reduced-basis (AMOR) weight tensors (see~\ref{si:3d}).
Note that each model is swept along its own hyperparameter axis; the counts are not matched in mode budget, since $m=32$ modes per axis retains $m^2(m/2+1)=17\,408$ Fourier coefficients per channel against $32$ reduced-basis functions for AMOR.
At equal $L^2$ reconstruction error for both inputs and outputs on this dataset, $m=32$ modes per axis for FNO corresponds to $k_{\mathrm{in}}=10$, $k_{\mathrm{out}}=5$ for AMOR, so the comparison shown is conservative.

Overall, AMOR's computational advantage over FNO is concentrated in the higher-mode, wider-network regime typically required for competitive accuracy in three dimensions.

\subsection{General Geometry Benchmarks}\label{subsec:irregular}

\subsubsection{Benchmark problems from Geo-FNO}
Standard FNO cannot be applied to the Airfoil and Elasticity problems because both involve irregular meshes with spatially varying node densities, precluding the use of the standard FFT.

We therefore compare AMOR against published results from Geo-FNO~\cite{li2023fourier} and DSE~\cite{lingsch2024beyond} as reference methods. Both approaches are explicitly designed to overcome the limitations of standard FNO on irregular meshes: Geo-FNO maps the irregular mesh to a latent uniform grid, whereas DSE employs a Fourier transform on non-equispaced points.

Results are collected in Table~\ref{tab:irregular}. In order to closely match the setup from Geo-FNO ~\cite{li2023fourier}, the training details are slightly different from the standard setup (see \ref{si:geofno} for details on the training setup).
\begin{table}[htbp]
\centering
\caption{Relative error [\%] (mean $\pm$ std over 5 seeds) on the airfoil dataset and the irregular elasticity problem (mesh size 972). AMOR reports both $L^1$ and $L^2$; other methods report only the metric available in their source. Best results in \textbf{bold}. AMOR achieves the lowest reported $L^2$ error on both the Airfoil and irregular-Elasticity datasets, outperforming Geo-FNO and other published baselines without requiring a structured-grid Fourier representation. Reported error norms follow the corresponding source methods.}
\label{tab:irregular}
\begin{tabular}{lcc}
\toprule
\textbf{Method} & \textbf{Airfoil} & \textbf{Elasticity} \\
\midrule
AMOR 
& $0.229 \pm 0.008$ ($L^1$)
& $\boldsymbol{0.98 \pm 0.02}$ ($L^1$)\\
 
& $\boldsymbol{0.957 \pm 0.017}$ ($L^2$)
& $\boldsymbol{1.32 \pm 0.02}$ ($L^2$) \\
\midrule
Geo-FNO~\cite{li2023fourier} 
& 1.38 ($L^2$) 
& 2.29 ($L^2$) \\
FNO interpolation~\cite{li2023fourier} 
& 4.21 ($L^2$) 
& 5.08 ($L^2$) \\
GraphNO~\cite{li2023fourier} 
& --- 
& 12.60 ($L^2$) \\
DeepONet~\cite{li2023fourier} 
& --- 
& 9.65 ($L^2$) \\
DSE~\cite{lingsch2024beyond} 
& \textbf{0.220} ($L^1$) 
& 1.73 ($L^1$) \\
\bottomrule
\end{tabular}

\end{table}
AMOR's average relative reduction in error over Geo-FNO -- $31\%$ on airfoil and $42\%$ on elasticity -- indicate that a data-driven reduced basis with deep cross-basis coupling remains competitive even on problems selected to showcase a structured-grid method's own geometry extension. On the Airfoil dataset, DSE performs slightly better than AMOR in relative $L^1$ error, while AMOR performs better on the Elasticity dataset. One possible reason is that AMOR minimizes the relative $L^2$ error of the Geo-FNO setup, whereas DSE is evaluated in relative $L^1$. The other published neural operator methods that accept irregular meshes as input (FNO interpolation, GraphNO, and DeepONet) all have significantly higher relative $L^2$ error.

\subsubsection{Elasticity of an aircraft rib}\label{subsec:rib}
We evaluate AMOR on the aircraft-rib tetrahedral mesh under three physically distinct loading regimes -- concentrated tip loading (\emph{cantilever\_tip}), distributed aerodynamic loading (\emph{distributed\_aero}), and localized multi-patch loading through fittings/brackets (\emph{multipatch\_bracket}) -- and their pooled union (\emph{mixed}, $40\,000$ samples). 

Standard FNO does not apply on this mesh; applying Geo-FNO would require an additional geometry-to-latent-domain mapping or domain-decomposition construction, which we do not pursue for this large 3D topology.
Message-passing graph operators remain applicable in principle, but with $n_{\mathrm{node}}=72{,}615$ (and $3n_{\mathrm{node}}$ displacement DOFs) their cost scales with the mesh graph rather than with a small reduced basis.
We therefore evaluate AMOR here and do not report Geo-FNO or other baselines on this problem.

Table~\ref{tab:rib} reports the resulting relative $L^1$/$L^2$ test errors.
\begin{table}[htbp]
\centering
\caption{AMOR test-set errors on the aircraft-rib mesh, per loading scenario and on the pooled mixed dataset. AMOR achieves good accuracy on all three load-cases, although the error increases with the complexity of the load scenario. FNO is not applicable to this geometry.}
\label{tab:rib}
\begin{tabular}{lcc}
\toprule
Scenario & $L^1$ [\%] & $L^2$ [\%] \\
\midrule
Cantilever tip load      & 0.04  & 0.04  \\
Distributed aero load    & 1.17  & 1.14  \\
Multipatch bracket load  & 10.40 & 10.94 \\
Mixed (pooled)           & 6.57  & 7.34  \\
\bottomrule
\end{tabular}
\end{table}
Error increases with loading complexity: sub-percent for the tip load, $1.17\%$ for the uniform aero load, and $10.4\%$ for the multipatch load, whose variable number of independently oriented patches makes the input--output map more variable. 
The best configurations retain only $k_{\mathrm{out}}\in\{2,5,30\}$ output modes against $3n_{\mathrm{node}}=217\,845$ displacement degrees of freedom, keeping both the offline basis construction and the online evaluation cheap.
A single model trained on the pooled dataset generalizes across all three regimes at a moderate accuracy cost relative to per-scenario models (Table~\ref{tab:rib}, last row), with multipatch loading remaining the dominant error source (\ref{si:rib}). 
Among the pre-computed bases tested on this mesh, PCA and Laplacian are the most accurate; the learned Cayley basis is slightly more accurate but costlier, and we use PCA throughout this section; the full comparison is reported in~\ref{si:rib_basis}.

\subsection{Transfer to unseen meshes}\label{subsec:transfer}
With the dual encoder of Section~\ref{subsec:meshes}, trained models are evaluated on unseen meshes without retraining. For nonlinear diffusion $(\ell_m,\ell_f)=(0.3,0.2)$ trained on a $100\times100$ grid, the relative $L^1$ test error is $2.00\%$, $1.87\%$, and $1.88\%$ on $50^2$, $150^2$, and $200^2$ grids, compared with $1.91\%$ on the training grid. The cantilever-tip rib model, transferred from its training mesh to three finer meshes, is more accurate than prolongation of the coarse finite-element solution on all three ($9.45$--$15.02\%$ versus $9.79$--$17.16\%$).
More details can be found in~\ref{si:mesh_generalization} and ~\ref{si:rib_mesh_transfer}.

\section{Limitations}\label{sec:limitations}

When AMOR uses a snapshot-derived basis, its main limitations trace back to that choice.

When the solution manifold requires many global linear modes -- transport-dominated problems, moving fronts, strong multiscale coupling -- the variance-based truncation that makes AMOR efficient can require so many modes that the efficiency of the reduced representation is lost. 
Kolmogorov flow illustrates this concretely. 
Even when the network predicts an Euler residual and the output PCA is fit to $\Delta u$, a fixed global basis still trails FNO's Fourier modes on this periodic, transport-dominated map. 
Closing the remainder likely needs dynamics-aware constructions beyond a single static spatial basis (trajectory-window PCA, jointly trained nonlinear encoders, or multi-step losses).
Learned and Laplacian bases narrow this gap but do not close it (Table~\ref{tab:basis}).
This is a limitation of the number of modes needed for a fixed global basis, not of sharp features per se: AMOR handles discontinuities well when a moderate number of modes captures them.

A related set of limitations concerns the basis's dependence on training data. 
Fitting input and output bases separately, while helpful when the two manifolds differ, can miss structure that is simple for the operator but not for individual snapshots, as in bifurcating or composed families.
Because $\mathcal{K}_{\mathrm{RB}}^{(\ell)}$ already couples all retained input and output modes (Eq.~\eqref{eq:transform}), this is a property of how the bases are fit rather than of the architecture, and a joint or coupled fitting objective -- e.g.\ a coupled eigenproblem or joint SVD over paired input--output snapshots, in place of separate PCA on each side -- is a direct extension within the existing basis taxonomy that we have not yet tested.
The fixed, data-independent bases (Laplacian, polynomial) do not need to be fit to training data at all, and are a natural fallback when training data is scarce.

Mesh transfer requires a one-time basis prolongation and construction of the target-mesh Gram-weighted dual encoder 
-- a step that FNO's resolution-independent Fourier basis does not require. 
In the settings studied here this step is cheap and transfer accuracy is strong, so it is a procedural cost rather than a failure of the design.

Finally, each reduced-basis block carries a dense coefficient tensor $T^{(\ell)}\in\R^{k_{\mathrm{dec}}\times k_{\mathrm{enc}}\times d_v\times d_v}$ (Table~\ref{tab:complexity}), which --- when not using the factorized version --- typically gives AMOR more parameters than FNO on the two-dimensional benchmarks studied here, where an adequate PCA truncation retains more modes than FNO's per-axis cutoff. 
Because the AMOR block does not grow with spatial dimension while FNO's mode tensor scales as $m^d$, this ordering narrows and can reverse in three dimensions; despite the larger count in two dimensions we observe no significant difference in training or inference time, though inference memory on constrained hardware can still favour FNO.

\section{Conclusions}\label{sec:outlook}

We introduce AMOR, an operator-learning architecture that replaces FNO's Fourier-based spectral convolution with reduced-basis blocks usable on arbitrary meshes, admits arbitrary and separate input/output bases through a common encode/decode interface, and, for bases that contain the constant function and separate points, is a universal approximator of continuous operators.
On benchmark problems, AMOR is competitive with FNO and its variants at comparable computational cost while extending directly to non-tensor product meshes.

One open question follows directly from our discussion: 
For time-dependent, transport-dominated dynamics, we view constructions that adapt to the dynamics -- PCA fit on trajectory windows or state increments, nonlinear encoders trained jointly with the operator, coupled learning of basis and operator, or multi-step losses -- as the necessary next step, rather than a better static spatial basis alone. Reducing the offline and coefficient-tensor cost further, via low-rank or tensor-structured coefficient maps and randomized or subsampled PCA fitting, would complement either direction.

Many industry-relevant problems require complex or adaptive meshes that structured-grid methods cannot accommodate; our aircraft-rib experiments show AMOR operating directly on such meshes — a highly irregular 3D finite-element mesh with tens of thousands of nodes — with no re-meshing and no change to the learned weights across resolutions. 
This makes AMOR a natural candidate for 3D industrial surrogate modeling (crashworthiness, patient-specific hemodynamics, additive-manufacturing process optimization) wherever an existing simulation database sits on a tetrahedral, hexahedral, or even completely unstructured mesh — where standard FFT-based FNO needs additional geometric handling.

\section*{Code and Data availability}
Code is publicly available at \url{https://github.com/leonardvetter/AMOR/}.\\
The nonlinear diffusion dataset, the 3D elastostatics dataset as well as the aircraft rib datasets were generated by the authors; the Helmholtz dataset was generated using a data generation pipeline
based on the \texttt{hIPPYflow} library~\cite{olearyroseberry2021hippyflow}.
All of the above are publicly available on Zenodo~\cite{zenodo2026datasets}.
The FlowBench dataset~\cite{tali2024flowBench} is publicly available at \url{https://huggingface.co/datasets/BGLab/FlowBench}
The Kolmogorov flow datasets at $\mathrm{Re}=40$, $500$, and $5000$
were introduced in~\cite{li2022mno} and are publicly available on
Zenodo~\cite{li2022mno_zenodo}.
The airfoil (NACA) and elasticity datasets were introduced
in~\cite{li2023fourier} and are available at
\url{https://drive.google.com/drive/folders/1YBuaoTdOSr_qzaow-G-iwvbUI7fiUzu8}.

\section*{Declaration of competing interest}
The authors declare that they have no known competing financial interests or personal relationships that could have appeared to influence the work reported in this paper.

\section*{Acknowledgements}
The authors acknowledge support by the state of Baden-Württemberg through bwHPC as well as the computational resources provided by the Ohio Supercomputer Centre (OSC).

\appendix

\bibliographystyle{elsarticle-num}
\bibliography{local.bib}

\section{Extended Theory}\label{si:extended_theory}

\subsection{FNO as a Neural Operator}\label{si:fno}
In the Fourier Neural Operator (FNO)~\cite{li2021fourier}, the kernel $\kappa^{(\ell)}$ is parameterized in Fourier space: the input is transformed via the Fast Fourier Transform (FFT), a matrix $\Hat{P}^{(\ell)}_{\mathbf k}\in\mathbb C^{d_v\times d_v}$ is applied to each retained Fourier mode ($m$ per axis), and the result is transformed back. Schematically, the FNO spectral convolution takes the form
\begin{equation}\label{eq:fno_layer}
    \big(\mathcal{K}_{\mathrm{FNO}}^{(\ell)} h\big)(x) = \mathcal{F}^{-1}\!\Big(\Hat{P}^{(\ell)} \cdot \big(\mathcal{F} h\big)\Big)(x),
\end{equation}
where $\mathcal{F}$ and $\mathcal{F}^{-1}$ denote the (truncated) Fourier transform and its inverse, and $\Hat{P}^{(\ell)}$ acts on each retained mode.
This parametrization exploits the convolution theorem to achieve global receptive fields efficiently, but the standard FFT implementation assumes an equispaced grid.
Geo-FNO~\cite{li2023fourier} relaxes this assumption by learning a deformation from the physical domain to a regular latent grid, on which the FNO acts.

\subsection{Comparison with universality results for FNO, PODNO and PCA-Net}\label{si:universality_comparison}
The universality of FNO was first established in~\cite{kovachki2021universal} for periodic domains, requiring an unbounded number of Fourier modes in the limit. The ANO framework of~\cite{lanthaler2025nonlocality} sharpened this to show that a \emph{single} Fourier mode (the spatial average) suffices, at the cost of increased channel width; a single averaging layer suffices.

For PODNO~\cite{cheng2026podno}, universality is established for the generalized spectral operator (GSO): adapting the constructive FNO proof of~\cite{kovachki2021universal}, GSO is shown to approximate continuous operators provided the activation is non-polynomial and globally Lipschitz and, critically, the basis functions defining the inverse reduced-basis transform, PODNO's $\Pi_N^{-1}$, lie in an appropriate Sobolev space, whereas Theorem~\ref{thm:AMOR_universality} instead requires continuous bases that separate points and contain the constant function.

For PCA-Net~\cite{bhattacharya2021model}, universality holds in the Bochner $L^2(\mu)$-sense (with respect to an input measure $\mu$) but is restricted to a single encoding-decoding step with nonlinearities only in the latent space.
Lanthaler~\cite{JMLR:v24:23-0478} strengthened this result, proving universality under minimal assumptions (continuity of the operator and finite second moments), but also established lower complexity bounds showing that PCA-Net faces a "curse of parametric complexity": for general Lipschitz operators, the required network size within the single encoding-decoding step can grow exponentially with the PCA truncation dimension; whether deeper compositions such as AMOR avoid this bound is open.
When considering a simple one layer AMOR with added nonlinearities inside the basis block, the architecture also inherits the universal approximation property of PCA-Net. 
However, adding nonlinearities inside of the reduced basis block also led empirically to worse performance on the considered datasets.

AMOR occupies a middle ground: like PCA-Net it can use data-driven bases on arbitrary discretizations, but like FNO it employs deep compositions of nonlocal and nonlinear layers.

\subsection{FiLM Conditioning}\label{si:film}
Feature-wise Linear Modulation~\cite{perez2017filmvisualreasoninggeneral} (FiLM) is defined in Section~\ref{subsec:amor_architecture}. 
FiLM does not change the reduced-basis kernel or the pointwise skip $W^{(\ell)}$.

On the aircraft-rib problems we use FiLM only for scalar/low-dimensional load metadata that is not already a nodal field:
\begin{itemize}
  \item \emph{Cantilever tip:} $\pi\in\R$ is the scalar tip load ($d_{\mathrm{film}}=64$).
  \item \emph{Distributed aero / multipatch bracket / mixed:} $\bm\pi$ concatenates Poisson ratio $\nu$ and a discrete case/patch indicator; the traction enters as a nodal input field.
\end{itemize}
Structured-grid benchmarks in Table~\ref{tab:benchmark} do not use FiLM.

\subsection{Proof and Technical Details for Universal Approximation}\label{si:universality}
We prove Theorem~\ref{thm:AMOR_universality} by separating the approximation
argument from its realization in AMOR.
The notation and assumptions of the main text remain in force.

\subsubsection{Notation and Standing Assumptions}\label{si:universality_notation}
We equip $X_C=\bar\Omega\times\{1,\ldots,C\}$ with the discrete topology on the
component index and the product of Lebesgue and counting measures. The boundary
has zero measure. Identifying $f\colon\Omega\to\mathbb R^C$ with
$(x,c)\mapsto f_c(x)$ gives an isometry $L^2(\Omega;\mathbb R^C)\cong L^2(X_C)$.
We abbreviate $X_{C_{\mathrm{in}}}$ and $X_{C_{\mathrm{out}}}$ by
$X_{\mathrm{in}}$ and $X_{\mathrm{out}}$.
For a $C$-component field, let
\[
 \langle f\rangle_C=\frac{1}{C|\Omega|}\sum_{c=1}^C\int_\Omega f_c(x)\,\mathrm dx,
 \qquad \nu_C=\sqrt{C|\Omega|},\qquad
 e_{\mathbf1,C}=\mathbf1/\nu_C.
\]
Averages of multichannel fields are taken channelwise. We omit $C$ when the space
is clear, and use $|\cdot|_\infty$ for the maximum norm of a finite-dimensional vector.
In particular, our convention for vector-valued inputs is
\[
 \|a\|_{L^\infty}
   =\operatorname*{ess\,sup}_{x\in\Omega}\max_{1\le c\le C_{\mathrm{in}}}|a_c(x)|.
\]

We use the mode-value maps from \eqref{eq:mode_map} and assume without loss of generality that $M_\infty>0$.
Throughout the proof, \eqref{eq:sigma_assump} supplies a globally Lipschitz,
nonpolynomial activation with $\sigma(0)=0$ that is strictly monotone on
a nonempty open interval $\mathcal I$. No differentiability of $\sigma$ is used.

\subsubsection{Approximation Tools}\label{si:universality_tools}
We use four tools: finite-rank approximation of a continuous operator on a
compact set, the ANO averaging result adapted from Lanthaler et al. \cite{lanthaler2025nonlocality}, a lemma that shows that the modes can act as positional encoding, and the standard density theorem for shallow
networks with continuous nonpolynomial activation~\cite{leshno1993multilayer}.
The latter applies on any compact subset of a finite-dimensional Euclidean space.

\begin{proposition}[Finite-rank reduction with continuous output functions]\label{prop:reduction}
Let $K\subset\mathcal U$ be compact and $\Psi^\dagger\colon K\to\mathcal V$ continuous. 
For every $\delta>0$, there are $J\in\mathbb N$, functions $\eta_1,\ldots,\eta_J\in C(\bar\Omega;\mathbb R^{C_{\mathrm{out}}})$, and continuous functionals $\alpha_j\colon L^1(\Omega;\mathbb R^{C_{\mathrm{in}}})\to\mathbb R$ such that
\begin{equation}\label{eq:via_step1}
 \sup_{a\in K}\left\|\Psi^\dagger(a)-\sum_{j=1}^J\alpha_j(a)\eta_j\right\|_{\mathcal V}
 \le\delta.
\end{equation}
\end{proposition}
The functions $\eta_j$ are auxiliary output functions, not modes of the fixed AMOR output basis. 
The nonlinear decoder will approximate them using mode values.
\begin{proof}
This is the $L^2$ version of \cite[Prop.~A.8]{lanthaler2025nonlocality},
adapted to a target that is continuous only on $K$.
Choose a countable subset of
$C(\bar\Omega;\mathbb R^{C_{\mathrm{out}}})$ that is dense in $\mathcal V$.
Gram--Schmidt, with dependent elements discarded, gives a complete
orthonormal system $(\eta_j)_{j\ge1}$ whose elements are continuous on
$\bar\Omega$.
Its orthogonal projections $\Pi_J$ converge strongly to the identity. 
Since $\Psi^\dagger(K)$ is compact and $\|\Pi_J\|\le1$, this convergence is uniform on $\Psi^\dagger(K)$. Choose $J$ so that $\sup_{a\in K}\|\Psi^\dagger(a)-\Pi_J\Psi^\dagger(a)\|\le\delta$.

On $K$, define $\alpha_j(a)=\langle\Psi^\dagger(a),\eta_j\rangle$.
These functionals are continuous in the $L^2$ topology. The identity map from $K$ with its $L^2$ topology to $K$ with its $L^1$ topology is a continuous bijection from a compact space to a Hausdorff space, hence a homeomorphism. 
Therefore the $\alpha_j$ are also $L^1$-continuous on $K$. 
The set $K$ is compact, and thus closed, in $L^1$. 
Tietze's theorem extends each $\alpha_j$ continuously to all of $L^1$.
On $K$, the finite sum in \eqref{eq:via_step1} equals exactly $\Pi_J\Psi^\dagger(a)$.
\end{proof}

The next proposition approximates a functional by a local feature map
$R\colon\mathbb R\times\overline{\Omega'}\to\mathbb R^{d_w}$,
an average, and a readout $\mathcal G\colon\mathbb R^{d_w}\to\mathbb R^J$.
Both maps and the feature dimension are chosen for the target and tolerance.

\begin{proposition}[Encoding by averages]\label{prop:ano_encoder}
Let $\Omega'\subset\mathbb R^d$ be a nonempty bounded, possibly disconnected, open set with Lipschitz boundary, $\alpha\colon L^1(\Omega')\to\mathbb R^J$ continuous, and $K'\subset L^1(\Omega')$ compact with $\sup_{u\in K'}\|u\|_{L^\infty}<\infty$. 
For every $\varepsilon_1>0$ there are $d_w\in\mathbb N$ and continuous maps $R\colon\mathbb R\times\overline{\Omega'}\to\mathbb R^{d_w}$, $\mathcal G\colon\mathbb R^{d_w}\to\mathbb R^J$ such that
\[
 \sup_{u\in K'}\Bigl|\alpha(u)-\mathcal G\Bigl(\tfrac{1}{|\Omega'|}\int_{\Omega'}R(u(z),z)\,\mathrm dz\Bigr)\Bigr|_\infty\le\varepsilon_1 .
\]
\end{proposition}
\begin{proof}
For $J=1$ this is \cite[Lemma~A.9, Remark~A.10]{lanthaler2025nonlocality}, whose proof extends inputs by zero to a bounding box and does not use connectedness. 
There, $R$ is a neural network and $\mathcal G=q_1\circ\sigma_0$ with a smooth, nonpolynomial, Lipschitz activation $\sigma_0$. 
We use only the continuity of $R$ and $\mathcal G$, so $\sigma_0$ need not equal $\sigma$; this admits nonsmooth activations such as ReLU in Theorem~\ref{thm:AMOR_universality}. 
For $J>1$, apply the case $J=1$ componentwise with tolerance $\varepsilon_1$ and concatenate the features.
\end{proof}

The next lemma shows why continuous separating mode values can replace explicit coordinates.

\begin{lemma}[Modes as positional encoding]\label{lem:positional}
Let $X$ be a compact metric space and $\Theta\colon X\to\mathbb R^{r_1}$ continuous
and injective. Then, every continuous $f\colon X\to\mathbb R^{r_2}$ can be written as
$f=g\circ\Theta$ for a continuous $g\colon\mathbb R^{r_1}\to\mathbb R^{r_2}$.
\end{lemma}
\begin{proof}
The map $\Theta\colon X\to\Theta(X)$ is a homeomorphism by compactness (and since $\mathbb R^{r_1}$ is a Hausdorff space).
The continuous function $f\circ\Theta^{-1}$ on the closed set $\Theta(X)$ extends
componentwise to a continuous $g$ on $\mathbb R^{r_1}$ by Tietze's extension theorem.
\end{proof}

\subsubsection{Approximation by Averaged Modal Features}\label{si:universality_features}
We use a small trick to extend the results to vector-valued inputs. Choose $M>0$ large enough so that the translated domains
$\Omega_c=\Omega+c (M,0,\ldots,0)$ have pairwise disjoint closures, and set
$\widetilde\Omega=\bigcup_{c=1}^{C_{\mathrm{in}}}\Omega_c$.
The map $z(y,c)=y+c(M,0,\ldots,0)$ identifies $X_{\mathrm{in}}$ with
$\overline{\widetilde\Omega}$. Define
\[
  \iota\colon L^1(\Omega;\mathbb R^{C_{\mathrm{in}}})
  \longrightarrow L^1(\widetilde\Omega),
  \qquad
  (\iota a)(z(y,c)):=a_c(y).
\]
This is a linear homeomorphism between the corresponding $L^1$ spaces and preserves
the chosen $L^\infty$ norm. Moreover,
\begin{equation}\label{eq:average_transfer}
 \frac1{|\widetilde\Omega|}\int_{\widetilde\Omega}R((\iota a)(z),z)\,\mathrm dz
   =\bigl\langle R(a(\cdot),z(\cdot))\bigr\rangle_{C_{\mathrm{in}}}.
\end{equation}
Thus Proposition~\ref{prop:ano_encoder} applies to vector-valued inputs by viewing
each component as a scalar function on its own copy of the spatial domain. The next lemma shows that the encoding by averages works also when using modes instead of an explicit coordinate input, given that the input basis separates points.

\begin{lemma}[Approximation by averaged modal features]\label{lem:modal_features}
Assume the input hypotheses of Theorem~\ref{thm:AMOR_universality}, and let
$\alpha\colon L^1(\Omega;\mathbb R^{C_{\mathrm{in}}})\to\mathbb R^J$ be continuous.
For every $\varepsilon_1>0$, there are finite integers $d_f,d_w$, coefficients
$\omega_i,\vartheta_i\in\mathbb R$, $\zeta_i\in\mathbb R^{k_{\mathrm{in}}}$,
$U\in\mathbb R^{d_w\times d_f}$, $\bar b\in\mathbb R^{d_w}$, a continuous map
$\mathcal G\colon\mathbb R^{d_w}\to\mathbb R^J$, and $w_{\max}>0$, such that for $i\in \{1,\dots,d_f\}$
\begin{equation}\label{eq:modal_features}
\begin{aligned}
 F_i(a)(y,c)&=\sigma\bigl(\omega_i a_c(y)+
                   \zeta_i\cdot\mathbf p^{\mathrm{in}}(y,c)+\vartheta_i\bigr),\\
 \widehat w(a)&=U\langle F(a)\rangle_{C_{\mathrm{in}}}+\bar b
                  \in[-w_{\max},w_{\max}]^{d_w},
\end{aligned}
\end{equation}
and
\begin{equation}\label{eq:via_step2}
 \sup_{a\in K}|\alpha(a)-\mathcal G(\widehat w(a))|_\infty\le\varepsilon_1.
\end{equation}
\end{lemma}
\begin{proof}
The set $\iota(K)$ is $L^1$-compact and uniformly bounded in $L^\infty$.
Apply Proposition~\ref{prop:ano_encoder} on $\widetilde\Omega$ to
$\alpha\circ\iota^{-1}$ with tolerance $\varepsilon_1/2$. By
\eqref{eq:average_transfer}, the resulting maps satisfy
\begin{equation}\label{eq:via_ano}
 w(a)=\langle R(a(\cdot),z(\cdot))\rangle_{C_{\mathrm{in}}},
 \qquad
 |\alpha(a)-\mathcal G(w(a))|_\infty\le\varepsilon_1/2.
\end{equation}

Next, we replace positions with mode values. Apply
Lemma~\ref{lem:positional} on the compact space
$[-M_\infty,M_\infty]\times X_{\mathrm{in}}$ to the continuous injection
$(\tau,y,c)\mapsto(\tau,\mathbf p^{\mathrm{in}}(y,c))$.
It gives a continuous
$\widetilde R\colon\mathbb R^{1+k_{\mathrm{in}}}\to\mathbb R^{d_w}$ such that
$R(\tau,z(y,c))=\widetilde R(\tau,\mathbf p^{\mathrm{in}}(y,c))$
for $|\tau|\le M_\infty$.

To control the feature approximation error, set
$\mathcal C=[-M_\infty,M_\infty]\times\mathbf p^{\mathrm{in}}(X_{\mathrm{in}})$,
$w_{\max}=1+\max_{(\tau,\xi)\in\mathcal C}|\widetilde R(\tau,\xi)|_\infty$, and
$\mathcal W=[-w_{\max},w_{\max}]^{d_w}$.
The maximum is finite by continuity and compactness, and
$|w(a)|_\infty\le w_{\max}-1$ for every $a\in K$.
Uniform continuity of $\mathcal G$ on $\mathcal W$ gives $\delta_R\in(0,1]$ such that
\[
 w,w'\in\mathcal W,\quad |w-w'|_\infty\le\delta_R
 \quad\Longrightarrow\quad |\mathcal G(w)-\mathcal G(w')|_\infty\le\varepsilon_1/2.
\]
Finally, to approximate the continuous $\widetilde R$, we use the density theorem for shallow networks to obtain $U,\bar b,\omega_i,\zeta_i,\vartheta_i$ with
\[
 \sup_{(\tau,\xi)\in\mathcal C}
 \left|U\bigl(\sigma(\omega_i\tau+\zeta_i\cdot\xi+\vartheta_i)\bigr)_{i=1}^{d_f}
          +\bar b-\widetilde R(\tau,\xi)\right|_\infty\le\delta_R.
\]
Averaging this bound gives $|\widehat w(a)-w(a)|_\infty\le\delta_R$.
The definition of $w_{\max}$ and $\delta_R\le1$ ensure that both vectors lie in
$\mathcal W$. Combining the last two bounds with \eqref{eq:via_ano} proves
\eqref{eq:via_step2}.
\end{proof}

\subsubsection{Realization in AMOR}\label{si:universality_realization}
The preceding lemma is an approximation statement. The next three lemmas explain
how the declared layers implement it in AMOR. Index the constant mode first in each basis.
We take zero offsets and unit gates; every unspecified weight and bias is zero. 
Because $\sigma(0)=0$, a residual block leaves an inactive channel unchanged, while the transition sets it to zero. The following simple lemma gives two operations that can be implemented in one layer of AMOR. 

\begin{lemma}[Two operations supplied by the constant modes]\label{lem:constant_operations}
Consider a reduced-basis kernel with first encoding and decoding modes
$\mathbf1/\nu_e$ and $\mathbf1/\nu_d$, respectively, where
$\nu_e=\sqrt{C_e|\Omega|}$ and $\nu_d=\sqrt{C_d|\Omega|}$.
It can perform each of the following operations exactly, before activation.
\begin{enumerate}
 \item From a constant source channel $h_{\cdot,j_0}=s\mathbf1$, $s\ne0$,
 it can write $\sum_n q_n\varphi_n^{\mathrm{dec}}$ into a chosen target channel $i_0$.
 \item From function-valued source channels $f_j$, it can write the constant channels
 $\sum_j U_{r,j}\langle f_j\rangle_{C_e}$ for any matrix $U$.
\end{enumerate}
\end{lemma}
\begin{proof}
The encoding of $s\mathbf1$ has first coefficient $\langle s\mathbf1,\mathbf1/\nu_e\rangle=s\nu_e$.
Only this coefficient is used, so orthogonality of the remaining modes to the constant is not needed.
For the first operation, set
\[
 T_{n,1,i_0,j_0}=q_n/(s\nu_e).
\]
The decoded coefficients are then exactly $q_n$.
For the second operation, use
\[
 T_{1,1,r,j}=(\nu_d/\nu_e)U_{r,j}.
\]
Indeed, encoding $f_j$ against the constant gives
$\nu_e\langle f_j\rangle_{C_e}$, and decoding the first mode divides by $\nu_d$.
All other tensor entries are zero in either construction.
\end{proof}
Next, we show that the encoding by averaging modal features from Lemma \ref{lem:modal_features} can be implemented by AMOR.
\begin{lemma}[Realization of the input encoder]\label{lem:input_realization}
Fix $F,U,\bar b,w_{\max}$ as supplied by
Lemma~\ref{lem:modal_features}. Set
$\mathcal W=[-w_{\max},w_{\max}]^{d_w}$ and recall
\[
 \widehat w:K\to\mathcal W,
 \qquad
 \widehat w(a)
 =U\langle F(a)\rangle_{C_{\mathrm{in}}}+\bar b.
\]
For a sufficiently large channel width, the lifting, layer $0$, and transition
can produce the constant channels
\begin{equation}\label{eq:encoded_state}
\begin{aligned}
 h^{(2)}_{\cdot,r}&=v_r(a)\mathbf1\quad(1\le r\le d_w),\\
 h^{(2)}_{\cdot,d_w+1}&=s_1\mathbf1,
\end{aligned}
\end{equation}
with all other channels zero, where $s_1\ne0$ and
$v(a)=\Upsilon(\widehat w(a))$. Here $\Upsilon$ is a homeomorphism from
$\mathcal W$ onto a compact set $\mathcal W_\sigma$.
\end{lemma}
\begin{proof}
\emph{Lifting.} Choose the first two channels as $a$ and $\mathbf1$ and all others
as zero; this is allowed by \eqref{eq:lifting}.

\emph{Input feature layer.} For each $i$, use the first operation in
Lemma~\ref{lem:constant_operations} to write
$\zeta_i\cdot\mathbf p^{\mathrm{in}}$ from the constant channel into channel $2+i$.
Choose the pointwise weight on the input channel to be $\omega_i$ and the bias
to be $\vartheta_i$. The residual in that target channel is zero, so the resulting
channel is exactly $F_i(a)$. The first two channels stay unchanged.

\emph{Transition.} Let $\mathcal I$ be a nonempty open interval of strict monotonicity supplied by \eqref{eq:sigma_assump}. Fix $t_0\in\mathcal I$ and choose $t>0$ such that
$\mathcal I_t=[t_0-tw_{\max},t_0+tw_{\max}]\subset\mathcal I$.
Set $W^{(1)}=0$. The second operation in Lemma~\ref{lem:constant_operations},
with matrix $tU$, writes $tU\langle F(a)\rangle$ into channels $1,\ldots,d_w$.
Add bias $t_0+t\bar b_r$ to channel $r$. After activation this gives
\[
 v_r(a)=\sigma(t_0+t\widehat w_r(a)).
\]
Use a bias $\vartheta_0$ in channel $d_w+1$ such that
$s_1=\sigma(\vartheta_0)\ne0$, which exists because $\sigma$ is not constant.
The transition has no residual connection, so all unused channels vanish.

Finally, define $\Upsilon(w)=(\sigma(t_0+tw_r))_{r=1}^{d_w}$.
Strict monotonicity makes $\Upsilon\colon\mathcal W\to\mathcal W_\sigma=\sigma(\mathcal I_t)^{d_w}$ a
homeomorphism, with continuous componentwise inverse
\[
 \Upsilon^{-1}(v)=\frac{(\sigma|_{\mathcal I_t})^{-1}(v)-t_0\mathbf1}{t}.
\]
Only local invertibility of the activation is used. This construction requires
$d_v\ge\max(d_f+2,d_w+1)$; additional zero channels can be reserved for decoding.
\end{proof}

\begin{lemma}[Realization of a continuous output decoder]\label{lem:output_realization}
Assume that the output basis satisfies the hypotheses of
Theorem~\ref{thm:AMOR_universality}. Let $\mathcal W_\sigma\subset\mathbb R^{d_w}$ be compact,
$\widetilde{\mathcal G}\colon\mathcal W_\sigma\to\mathbb R^J$ continuous, and
$\eta_1,\ldots,\eta_J\in C(X_{\mathrm{out}})$.
Suppose $h^{(2)}(v)$ consists of the constant channels $v_1,\ldots,v_{d_w}$,
a nonzero constant channel $s_1$, and otherwise zeros.
For every $\varepsilon_2>0$, one residual output block and a projection
$\mathcal Q$ of the form \eqref{eq:lifting} can be chosen so that
\begin{equation}\label{eq:via_step3}
 \sup_{v\in\mathcal W_\sigma}\left\|\mathcal Q(h^{(3)}(v))-
             \sum_{j=1}^J\widetilde{\mathcal G}_j(v)\eta_j\right\|_{\mathcal V}\le\varepsilon_2.
\end{equation}
It suffices to take $d_v\ge d_w+1+k_{\mathrm{out}}$, with a suitable finite $d_q$.
\end{lemma}
\begin{proof}
\emph{Output positional features.} Let
$\bar\varphi_{\mathrm{out}}=\max_{n,x,c}|\varphi_{n,c}^{\mathrm{out}}(x)|$.
Choose $t_0\in\mathcal I$ and $t'>0$ such that
$[t_0-t'\bar\varphi_{\mathrm{out}},t_0+t'\bar\varphi_{\mathrm{out}}]\subset\mathcal I$, and define
\[
 \Lambda(x,c)=\bigl(\sigma(t_0+t'\varphi_{n,c}^{\mathrm{out}}(x))\bigr)_{n=1}^{k_{\mathrm{out}}}.
\]
This map is continuous and injective because $\mathbf p^{\mathrm{out}}$ is
injective and $\sigma$ is strictly monotone on the chosen interval.
For $i_n=d_w+1+n$, use source channel $d_w+1$, carrying $s_1\mathbf1$, and the first
operation in Lemma~\ref{lem:constant_operations} to write
$t'\varphi_n^{\mathrm{out}}$ into channel $i_n$. Add bias $t_0$.
The target channel has zero residual, so after activation
$h^{(3)}_{c,i_n}(x)=\Lambda_n(x,c)$. All other updates are zero; hence the residual
preserves the channels carrying $v$ and $s_1$.

\emph{Shallow projection.} Lemma~\ref{lem:positional} gives continuous functions
$\chi_j$ with $\eta_j=\chi_j\circ\Lambda$.
The finite-dimensional function
\[
 f_{\mathcal Q}(v,\xi)=\sum_{j=1}^J\widetilde{\mathcal G}_j(v)\chi_j(\xi)
\]
is continuous on the compact set $\mathcal W_\sigma\times\Lambda(X_{\mathrm{out}})$.
Shallow-network density therefore supplies the parameters of $\mathcal Q$ in
\eqref{eq:lifting}, with zero input weights for all unused channels, such that
\[
 \left|\mathcal Q(h^{(3)}(v))(x)_c-f_{\mathcal Q}(v,\Lambda(x,c))\right|
   \le\frac{\varepsilon_2}{\sqrt{C_{\mathrm{out}}|\Omega|}}.
\]
Since $f_{\mathcal Q}(v,\Lambda(x,c))=\sum_j\widetilde{\mathcal G}_j(v)\eta_j(x,c)$, integration over
$X_{\mathrm{out}}$ proves \eqref{eq:via_step3}.
\end{proof}

\subsubsection{Proof of the Main Theorem}\label{si:universality_proof}
\begin{proof}[Proof of Theorem~\ref{thm:AMOR_universality}]
Fix $\varepsilon>0$ and let $\varepsilon_0=\varepsilon/2$.
Proposition~\ref{prop:reduction}, with $\delta=\varepsilon_0/3$, gives continuous
output functions $\eta_j$ and input functionals $\alpha_j$ satisfying
\eqref{eq:via_step1}. Define
\[
 \bar\eta=1+\sum_{j=1}^J\|\eta_j\|_{\mathcal V},
 \qquad \varepsilon_1=\frac{\varepsilon_0}{3\bar\eta}.
\]

Apply Lemma~\ref{lem:modal_features} to $\alpha=(\alpha_1,\ldots,\alpha_J)$ with
tolerance $\varepsilon_1$. Lemma~\ref{lem:input_realization} realizes its averaged
features as $v(a)=\Upsilon(\widehat w(a))$ in the transition output.
The map $\widetilde{\mathcal G}=\mathcal G\circ\Upsilon^{-1}$ is continuous on $\mathcal W_\sigma$, and
$\widetilde{\mathcal G}(v(a))=\mathcal G(\widehat w(a))$.
Apply Lemma~\ref{lem:output_realization} to this $\widetilde{\mathcal G}$ and the functions $\eta_j$,
with $\varepsilon_2=\varepsilon_0/3$.

The resulting network satisfies, uniformly for $a\in K$,
\[
\begin{aligned}
 \|\Psi^\dagger(a)-\Psi_\theta(a)\|_{\mathcal V}
 &\le\frac{\varepsilon_0}{3}
    +\sum_{j=1}^J|\alpha_j(a)-\mathcal G_j(\widehat w(a))|\,\|\eta_j\|_{\mathcal V}
    +\frac{\varepsilon_0}{3}\\
 &\le\frac{\varepsilon_0}{3}
    +\varepsilon_1\sum_{j=1}^J\|\eta_j\|_{\mathcal V}
    +\frac{\varepsilon_0}{3}
 \le\varepsilon_0<\varepsilon.
\end{aligned}
\]
The construction uses three hidden operator layers, indexed $0,1,2$, so $L=2$. 
For $L>2$, append $L-2$ residual output blocks with zero parameters; they act as the identity because $\sigma(0)=0$.
One may take $d_v=\max(d_f+2,\,d_w+1+k_{\mathrm{out}})$ and $d_q$
from Lemma~\ref{lem:output_realization}.
If the widths are tied, replace both by
$\max(d_f+2,\,d_w+1+k_{\mathrm{out}},\,d_q)$ and pad unused channels
and projection units with zeros. 
Continuity of $\Psi_\theta\colon\mathcal U\to\mathcal V$ follows from \eqref{eq:sigma_assump} and the boundedness of the linear maps.
\end{proof}

\begin{remark}[Factorized AMOR is also universal]\label{rem:factorized_universality}
The construction uses only encoding mode $m=1$.
For each output channel $i$, the matrix $(T^{(\ell)}_{n,1,i,j})_{n,j}$ has rank
at most one: the positional-feature operations use one source channel, and the
averaging operation uses one decoding mode. Hence these tensors admit the
factorization in Remark~\ref{rem:factorized}. The theorem therefore also holds for the factorized architecture.
\end{remark}

\begin{remark}[Reference offsets and gates]\label{rem:offsets_univ}
Theorem~\ref{thm:AMOR_universality} holds for the affine path of Remark~\ref{rem:offsets}. Let $L_\sigma$ be the Lipschitz constant of $\sigma$.
(i) \emph{Encoder offsets.} The construction uses only the first encoding mode, on which $h-\mu_{\mathrm{enc}}$ has coefficient $\nu_e(\langle h\rangle-\langle\mu_{\mathrm{enc}}\rangle)$. Operation~1 of Lemma~\ref{lem:constant_operations} therefore holds with $s$ replaced by $s-\langle\mu_{\mathrm{enc}}\rangle$, which is nonzero for a suitable lifting bias and $\vartheta_0$ because $\sigma$ is nonconstant; operation~2 acquires a constant that the bias absorbs. With $\mu_{\mathrm{dec}}=0$, the network $\Psi^0_\theta$ thus realizes the construction exactly.
(ii) \emph{Decoder offsets.} Set all gates to $g\in(0,1]$ and divide $W^{(\ell)},T^{(\ell)},b^{(\ell)}$ by $g$. The pre-activations then equal those of $\Psi^0_\theta$ plus $g\mu_{\mathrm{dec}}$ in every channel. In these rescaled parameters, each block is Lipschitz on its $L^2$ state space with a constant independent of $g$ and differs from its counterpart in $\Psi^0_\theta$ by at most $L_\sigma g\sqrt{d_v}\|\mu_{\mathrm{dec}}\|$. Hence $\sup_{a\in K}\|\Psi_\theta(a)-\Psi^0_\theta(a)\|_{\mathcal V}\le Cg$ with $C$ independent of $g$, and small $g$ preserves the bound of Theorem~\ref{thm:AMOR_universality}.
\end{remark}

\begin{remark}[Where the hypotheses enter]
Compactness of $K$ and continuity of $\Psi^\dagger$ give the finite-rank
reduction. The uniform $L^\infty$ bound permits the averaging encoder
and supplies a compact interval of input values.
Continuous separating modes replace spatial coordinates and component
labels; the constant modes implement averaging and mode generation.
Nonpolynomiality gives shallow-network density. Local strict
monotonicity preserves information across the transition activation
and the output positional map. The condition $\sigma(0)=0$ preserves
inactive residual channels. Global Lipschitz continuity ensures that
the full network is continuous between the stated $L^2$ spaces.
\end{remark}

\subsubsection{Scope and Basis Choice}\label{si:universality_scope}
The theorem keeps both bases fixed and lets $d_v$ grow.
Under its assumptions, fixed bases allow arbitrarily accurate
approximation, but the result gives no bound on the required width,
conditioning, or training error. The proof uses constant modes for
averaging and broadcasting, and mode values as positional features.
It therefore does not quantify how basis size or quality affects
approximation efficiency.

For scalar fields, any span containing $1,x_1,\ldots,x_d$ separates
points; for $C$ components, the span of
$e_c,x_1e_c,\ldots,x_de_c$, $c=1,\ldots,C$, does the same.
Here $e_c$ is the $c$th standard unit vector of $\mathbb R^C$.
These spans can be orthonormalized with the normalized global constant
first. Point separation with a constant mode requires $k\ge d+1$
by invariance of domain~\cite[Thm.~2B.3]{hatcher2002algebraic},
but this dimension bound alone is not sufficient.

POD, DIS, and learned bases adapt to the data distribution and,
depending on their construction, to the target operator.
This adaptation is intentional: distinctions that are irrelevant
for the task may be suppressed. For example, on a reflection-symmetric domain, even output modes
identify reflected positions, which is compatible with targets
whose outputs are even functions.
More generally, adapting to the data can reflect useful invariances
(ignoring irrelevant changes) or support equivariances (transforming
outputs consistently with transformed inputs).
Learning a basis alone does not guarantee these properties:
the modal weights and other layers must also respect the symmetry.

Such bases need not satisfy the theorem's constant-mode and separation
assumptions. Failure of these sufficient conditions does not rule out
accurate approximation of a particular target, but requires a separate
analysis. Distinct mode vectors on a sampled mesh establish separation
only on that mesh.

Purely periodic modes identify opposite faces of a closed period cell,
although they may separate points on the corresponding torus.
Bases whose every mode vanishes on the boundary lack the constant mode.
These cases require a separate argument; failure of a hypothesis here
is not an $L^2$ impossibility result.

For input laws with unbounded support, extending the result to
approximation in expectation requires suitable regularity and
integrability assumptions and a cut-off argument;
see \cite[Rem.~1]{lanthaler2025nonlocality}.

\subsubsection{Relation to Other Constructions}\label{si:universality_context}
The averaging result of \cite{lanthaler2025nonlocality} uses local maps
that can access position. Its application to eigenbasis operators in
\cite[Sec.~3.5]{lanthaler2025nonlocality}, including the NORM setting
\cite{chen2024norm}, assumes that averaging can be represented and that
the lifting and projection provide the required positional information.
AMOR's componentwise lifting and projection receive no coordinates.
The additional construction in Lemmas~\ref{lem:input_realization}
and~\ref{lem:output_realization} generates that information internally.
The separating-mode and constant-mode hypotheses are sufficient for
this construction; no necessity claim is made for all architectures.

\textbf{Fourier multipliers and positional information.}
On a torus, a standard FNO with frequency-diagonal multipliers,
spatially shared pointwise maps, and no coordinates, spatially varying
biases, or other positional features is translation-equivariant
\cite[Sec.~2.4]{lanthaler2025nonlocality}.
In particular, it maps a constant input to a constant output.
Although suitable Fourier modes separate points on the torus,
frequency-diagonal multipliers cannot turn a constant channel into
nonconstant positional features. AMOR permits off-diagonal couplings
from the constant encoding mode to the other decoding modes.
This distinction concerns the specified architecture, not the Fourier
basis itself.

\textbf{Fixed linear encoders and decoders.}
In the universal construction above, component interactions are formed
through the joint vector of averaged features. The full architecture
also permits component mixing in $W^{(1)}$.
For comparison, a construction of the form
\[
 a\longmapsto\bm\varphi^{\mathrm{out}}
 F\bigl((\bm\varphi^{\mathrm{in}})^*a\bigr)
\]
has a fixed linear input encoder and a fixed linear output decoder.
Its output lies in the output reduced space, and inputs with equal
encoded coefficients are indistinguishable.
These are restrictions of this explicit form.
They should be distinguished from truncation steps used within
approximation proofs
\cite{bhattacharya2021model,kovachki2021universal,cheng2026podno};
such proof steps do not by themselves imply that every network in
the corresponding architecture has a fixed linear information
bottleneck or output range.
AMOR applies nonlinear pointwise processing before averaging and
nonlinear reconstruction after decoding.
Removing all pointwise and residual routes around the initial linear
encoding creates an input-information bottleneck, but nonlinear
reconstruction still need not lie in the output basis span.

\section{Computational Details}\label{si:computational_details}

\subsection{Dataset generation details}\label{si:data}

\subsubsection{Nonlinear diffusion.} Discretized by a cell-centred finite-volume scheme on a uniform $100\times100$ grid (FiPy) and solved with the default sparse direct backend. Three versions use different Gaussian-process variance for the input fields, giving different input smoothness levels. Additional resolutions $50$--$200$ are generated for the mesh-generalization tests, and a further $10{,}000$-sample dataset is generated for memory-efficient-training experiments.

\subsubsection{Helmholtz.} The coupled real--imaginary state with perfectly matched layer (PML) absorption is discretized with continuous piecewise-linear (P1) Lagrange elements on a triangulated $100\times100$ rectangle covering the PML-extended domain (FEniCS/hippylib) and solved with a PETSc sparse direct (LU) backend. The fixed point source and the three $(\sigma,\rho)$ prior variants are otherwise identical across datasets.

\subsubsection{Kolmogorov flow.} All three datasets are supplied by~\cite{li2022mno,li2022mno_zenodo}.
All models predict a one-step Euler
residual, $\hat u(t+\Delta t)=u(t)+\Psi_\theta\bigl(u(t)\bigr)$, trained against
$u(t+\Delta t)$. 
The output basis is fit to
increments $\Delta u=u(t+\Delta t)-u(t)$; the input basis remains snapshot-based.

\subsubsection{Geo-FNO datasets: Elasticity and Airfoil}\label{si:geofno}
We consider the elasticity and airfoil datasets introduced in Geo-FNO~\cite{li2023fourier}, using the published training/test splits. The inputs are the nodal coordinates, and the targets are the stress field on the 972-node elasticity mesh and the Mach number on the $221\times51$ airfoil C-mesh. Coordinates are standardized per node and channel using training-split statistics, with standard deviations floored at $1\%$ of the corresponding channel's median nodal standard deviation. Outputs are normalized channelwise. All normalization statistics and mass-weighted PCA bases are computed from the training split only. The PCA inner product uses lumped nodal masses, obtained from dual-cell areas for the airfoil mesh and Delaunay-triangle areas for the elasticity point cloud.

\subsubsection{Linear elastostatics (3D).} Discretized with trilinear hexahedral ($Q_1$) elements on a uniform $64^3$ nodal grid ($63^3$ cells, $3\cdot64^3=786{,}432$ displacement DOFs) and solved with PETSc conjugate gradients preconditioned by algebraic multigrid. The modulus is a clipped log-normal field, $E=\mathrm{clip}(E_0\exp(\sigma g),0.2,5)$ with $E_0=1$, $\sigma=0.5$, and $g$ an eight-term truncated KL expansion of a squared-exponential field with correlation length $0.35$; Poisson ratio $\nu=0.3$ is fixed. The cube is clamped at $z=0$, loaded by a uniform traction $t_0=1$ at $z=1$, and traction-free laterally; boundary conditions are identical across samples.

\subsubsection{FlowBench LDC.} 1000 steady 3D incompressible Navier--Stokes lid-driven-cavity simulations on a $128^3$ grid with varying internal geometry (ellipsoids, toroids, boxes, cylinders), each a fully resolved DNS from the FlowBench simulator~\cite{tali2024flowBench}.

\subsubsection{Aircraft rib.} Tetrahedral mesh: $n_{\mathrm{tet}}=229{,}884$ cells, $135{,}040$ exterior triangular facets, $3n_{\mathrm{node}}=217{,}845$ displacement DOFs. The Lam\'e parameters derive from a log-normal Young's modulus field (nodal mean $E\approx7.6$--$9.0\times10^4$, correlation length $20$--$180$) and Poisson ratio $\nu$ ($0.3$ fixed for the cantilever tip variant; $\nu\sim\mathcal U[0.25,0.35]$ for the distributed aero and multipatch bracket variants). A fixed Dirichlet clamp $\mathbf u=\mathbf 0$ is imposed on a marked facet set, motivated by cantilever test configurations for aircraft ribs and spars~\cite{Petrasinovic2017,IJRSIWingSpar,Esposito2023}; a Neumann traction $\mathbf t=\hat{\mathbf d}(F/A_\Gamma)$ is applied on the loaded facets $\Gamma$, with the remainder traction-free. Loaded forces are clipped to $F\in[5,3000]$, spanning lightly loaded to concentrated fitting loads motivated by experimental ranges~\cite{Petrasinovic2017}. The multipatch-bracket patches are generated by nearest-centroid region growth to a target area fraction $\in[0.005,0.08]$, each with an independently sampled outward-oriented loading direction, with multiple simultaneous patches emulating several discrete attachment loads by linear superposition~\cite{IJERTBracket}. The forward problem is discretized with continuous piecewise-linear (P1) vector Lagrange elements in FEniCS/dolfin and solved with a PETSc sparse direct (LU) backend, with a relative-residual diagnostic recorded per solve. Each sample stores the physics-native facet traction field \texttt{traction\_facet}, a deterministic area-weighted boundary-nodal load field \texttt{load\_field\_boundary} ($\mathbf t_i=\sum_f A_f\mathbf t_f/\sum_f A_f$ over incident facets), the displacement field and its magnitude, and the modulus field; for the operator-learning models the inputs are the boundary load field together with $E(\mathbf x)$ and the target is $\mathbf u$.

\subsection{Architecture and Hyperparameters}\label{si:params}
AMOR and FNO share the number of nonlocal layers and the channel width $d_v$; in AMOR, $d_v$ counts channels per physical component, so the hidden state has $C\,d_v$ channels for $C$ components. 
Their nonlocal operators differ: FNO learns a $d_v\times d_v$ matrix for each retained Fourier mode, whereas AMOR learns a coefficient tensor that couples all pairs of retained modes, either dense or factorized (Remark~\ref{rem:factorized}). Parameter counts therefore differ and are listed in Table~\ref{tab:model_configs}.

\par \textbf{FNO.} Official implementation from the neuraloperator repository~\cite{kossaifi2025librarylearningneuraloperators}.

\par \textbf{PCANet.} Bhattacharya et al.~\cite{bhattacharya2021model} single-layer architecture, with no nonlinearity in physical space between projection and reconstruction; given the same $k_{\mathrm{in}}/k_{\mathrm{out}}$ and basis as AMOR wherever both are compared.

\par \textbf{AMOR.} Factorized version for all experiments where $k_{\mathrm{enc}}k_{\mathrm{dec}}>10^4$ (all except Aircraft Rib and Elastostatics), dropout $0$, batch normalization disabled, coordinate features disabled, constant mode disabled, reference offsets: training-split means, enabled in all experiments. 
Mode count selection is fixed at $144/144$ (matched to FNO) on the three 2D benchmarks; $99.9\%$ explained-variance truncation on Elastostatics 3D ($38/15$) and the Geo-FNO benchmarks.

On the three 2D benchmarks we fix $k_{\mathrm{in}}=k_{\mathrm{out}}=144$ to match FNO's nominal $12\times12$ truncation.
On Elastostatics 3D and the Geo-FNO problems we truncate PCA to $99.9\%$ explained variance on the training split. For training AMOR on the Geo-FNO datasets, we use a batch size of $20$ and the relative $L^2$ loss, following the Geo-FNO setup~\cite{li2023fourier}. Models are trained for 500 epochs without early stopping, and the final-epoch model is evaluated.
PCA on the FlowBench inputs is near-degenerate (a few leading modes already saturate practical variance cutoffs without resolving the flow; Sec.~\ref{subsec:3d}), so AMOR uses a Laplacian eigenbasis on the $128^3$ grid.
We select $(k_{\mathrm{in}},k_{\mathrm{out}})$ by a short capacity sweep over fixed mode pairs (with $k_{\mathrm{in}}$ clamped by the attainable input rank, $\approx 360$), scoring candidates by validation relative $L^1$; the reported configuration is $k_{\mathrm{in}}/k_{\mathrm{out}}=360/100$.
For all rib scenarios we select $(k_{\mathrm{in}},k_{\mathrm{out}})$, width, and depth by validation relative $L^1$ with early stopping; bases are fit on the training split only.
Scout budgets differ by scenario rather than by a common full grid.
The tip scout draws a random subset of $\{\text{width}\!\in\!\{16,24,32\}\}\times\{\text{layers}\!\in\!\{3,4,5\}\}\times \{(k_{\mathrm{in}},k_{\mathrm{out}})\in\{(15,2),(28,3),(40,5),(100,30),(200,50)\}\}$ and keeps the validation best ($100/30$, width~$16$, $5$ layers), later retrained on the full tip set.
For distributed aero, multipatch bracket, and mixed, candidates are a shorter ladder of ranks near PCA explained-variance cutoffs (plus nearby width/depth); we retain the validation-best settings in Table~\ref{tab:model_configs} ($15/2$, $30/5$, and $15/2$).
All reduced bases are fit on the training split only.

\begin{table}[htbp]
\centering
\small
\caption{Architecture hyperparameters used in the results for FNO, AMOR and PCANet.
For FNO, modes are retained Fourier modes per axis; for AMOR and PCANet they are the number of reduced-basis coefficients $(k_{\mathrm{in}}/k_{\mathrm{out}})$.
For PCANet, width and layers denote the MLP hidden dimension and depth.
Unless noted, training uses AdamW (learning rate $2\times10^{-3}$, weight decay $10^{-4}$), cosine annealing, grad-clip~$3$, relative $L^1$ loss, and early stopping.
Batch sizes are chosen based on one-sample memory and training hardware, resulting in 32 for nonlinear diffusion and Kolmogorov, 8 for Helmholtz, 1 for 3D elastostatics and FlowBench, 64 for the cantilever-tip aircraft rib version, 8 for the other aircraft rib versions, and 20 for the Geo-FNO datasets.}
\label{tab:model_configs}
\begin{tabular}{llcccc}
\toprule
Dataset & Model & Modes & Width & Layers & Parameters \\
\midrule
\multirow{3}{*}{Nonlinear diffusion (all variants)}
  & FNO    & $12\times12$ & $24$ & $4$ & 335{,}065\\
  & AMOR   & $144/144$    & $24$ & $4$ & 2{,}326{,}089\\
  & PCANet & $144/144$    & $1024$ & $5$ & 4{,}494{,}480 \\
\midrule
\multirow{3}{*}{Helmholtz $(\sigma,\rho)=(0.1,0.6)$}
  & FNO    & $12\times12$ & $20$ & $3$ & 174{,}682\\
  & AMOR   & $144/144$    & $20$ & $3$ & 1{,}419{,}129\\
  & PCANet & $144/144$    & $768$ & $4$ & 1{,}993{,}872\\
\midrule
\multirow{3}{*}{Helmholtz $(\sigma,\rho)=(0.2,0.4)$}
  & FNO    & $12\times12$ & $20$ & $3$ & 174{,}682\\
  & AMOR   & $144/144$    & $20$ & $3$ & 1{,}419{,}129\\
  & PCANet & $144/144$    & $768$ & $4$ & 1{,}993{,}872\\
\midrule
\multirow{3}{*}{Helmholtz $(\sigma,\rho)=(0.4,0.2)$}
  & FNO    & $12\times12$ & $24$ & $4$ & 335{,}066\\
  & AMOR   & $144/144$    & $24$ & $4$ & 2{,}326{,}114\\
  & PCANet & $144/144$    & $1024$ & $5$ & 4{,}494{,}480\\
\midrule
\multirow{3}{*}{Kolmogorov $\mathrm{Re}=40$}
  & FNO    & $12\times12$ & $24$ & $4$ & 335{,}041\\
  & AMOR   & $144/144$    & $24$ & $4$ & 2{,}325{,}512\\
  & PCANet & $144/144$    & $1024$ & $5$ & 4{,}494{,}480\\
\midrule
\multirow{3}{*}{Kolmogorov $\mathrm{Re}=500$}
  & FNO    & $12\times12$ & $32$ & $5$ & 744{,}033\\
  & AMOR   & $144/144$    & $32$ & $5$ & 4{,}061{,}481\\
  & PCANet & $144/144$    & $1024$ & $5$ & 4{,}494{,}480\\
\midrule
\multirow{3}{*}{Kolmogorov $\mathrm{Re}=5000$}
  & FNO    & $12\times12$ & $40$ & $5$ & 1{,}162{,}361\\
  & AMOR   & $144/144$    & $40$ & $5$ & 5{,}309{,}169\\
  & PCANet & $144/144$    & $1024$ & $5$ & 4{,}494{,}480\\
\midrule
\multirow{2}{*}{FlowBench LDC 3D}
  & FNO    & $12\times12\times12$ & $24$ & $4$ & 2{,}330{,}236\\
  & AMOR   & $360/100$ & $24$ & $4$ & 15{,}916{,}314\\
\midrule
\multirow{2}{*}{Elastostatics 3D}
  & FNO    & $12\times12\times12$ & $24$ & $4$ & 2{,}330{,}139\\
  & AMOR   & $38/15$\,($99.9\%$ var.) & $24$ & $4$ & 1{,}423{,}548\\
\midrule
\multirow{1}{*}{Airfoil (Geo-FNO)}
  & AMOR   & $6/153$\,($99.9\%$ var.) & 76 & 6 & 34{,}351{,}581\\
\midrule
\multirow{1}{*}{Elasticity (Geo-FNO)}
  & AMOR   & $19/106$\,($99.9\%$ var.) & 53 & 4 & 1{,}546{,}403\\
\midrule
\multirow{1}{*}{Rib cantilever tip}
  & AMOR   & $100/30$ & $16$ & $4$ & 307{,}532\\
\multirow{1}{*}{Rib distributed aero}
  & AMOR   & $15/2$ & $24$ & $4$ & 167{,}182\\
\multirow{1}{*}{Rib multipatch bracket}
  & AMOR   & $30/5$ & $32$ & $5$ & 1{,}200{,}791\\
\multirow{1}{*}{Rib mixed}
  & AMOR   & $15/2$ & $24$ & $4$ & 167{,}182\\
\bottomrule
\end{tabular}
\end{table}

\subsection{Training Hardware}\label{si:hardware}
We always use the same computational resources for all models on each dataset, NVIDIA V100-16GB GPUs for the nonlinear diffusion and the Geo-FNO datasets, Tesla V100S-PCIE-32GB for the Helmholtz, Kolmogorov, Elastostatics, FlowBench and aircraft rib datasets.

\section{Additional Results}\label{si:additional_results}

\subsection{Mesh generalization on the nonlinear diffusion dataset}\label{si:mesh_generalization}
We train one AMOR model on the nonlinear diffusion dataset ($n_s=7000$ training samples, length scales $(\ell_m, \ell_f)=(0.3, 0.2)$).

The native training resolution is $100\times100$. The model is transferred without retraining to $50\times50$, $150\times150$, and $200\times200$ grids by interpolating the decoder modes, $B=P\Phi_1$, and using the dual encoder of Section~\ref{subsec:meshes}, which for equal weights is $E=(B^\top B)^{-1}B^\top$. Errors are relative $L^1$ errors of the denormalized fields on a separate $1000$-sample test set at each resolution (Fig.~\ref{fig:mesh_independence}). The error varies between $1.87\%$ and $2.00\%$ across the four grids and does not increase away from the training resolution.
\begin{figure}[htbp]
    \centering
    \includegraphics[width=0.7\linewidth]{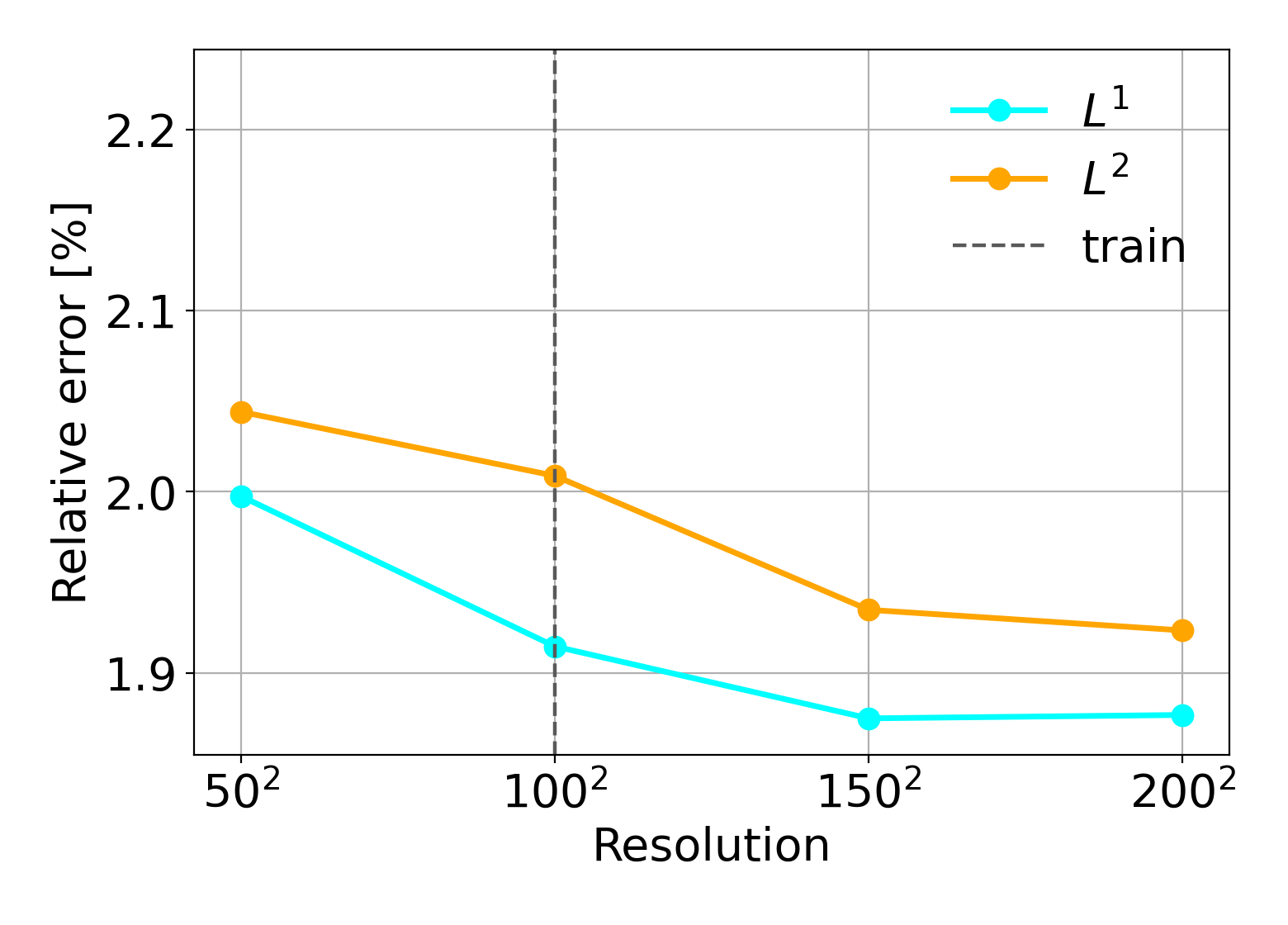}
    \caption{Mesh transfer experiment on the nonlinear dataset. Evaluation of an AMOR model trained on resolution $100\times100$ on unseen samples from the same and other resolutions. AMOR generalizes well to higher resolutions.}
    \label{fig:mesh_independence}
\end{figure}

\subsection{Visualization of Representative Model Predictions}\label{si:visualization}

We complement the quantitative benchmark results with predictions and
spatial error maps for selected test samples.

\textbf{Helmholtz.}
Figure~\ref{fig:helm_pred} compares FNO, AMOR, and PCA-Net predictions
with the reference solution for one randomly selected test sample from
each of the three Helmholtz variants.
Figure~\ref{fig:helm_err} shows pointwise relative absolute errors for
the three models, allowing their spatial error patterns to be compared.

\begin{figure}[htbp]
    \centering
    \includegraphics[width=\linewidth]{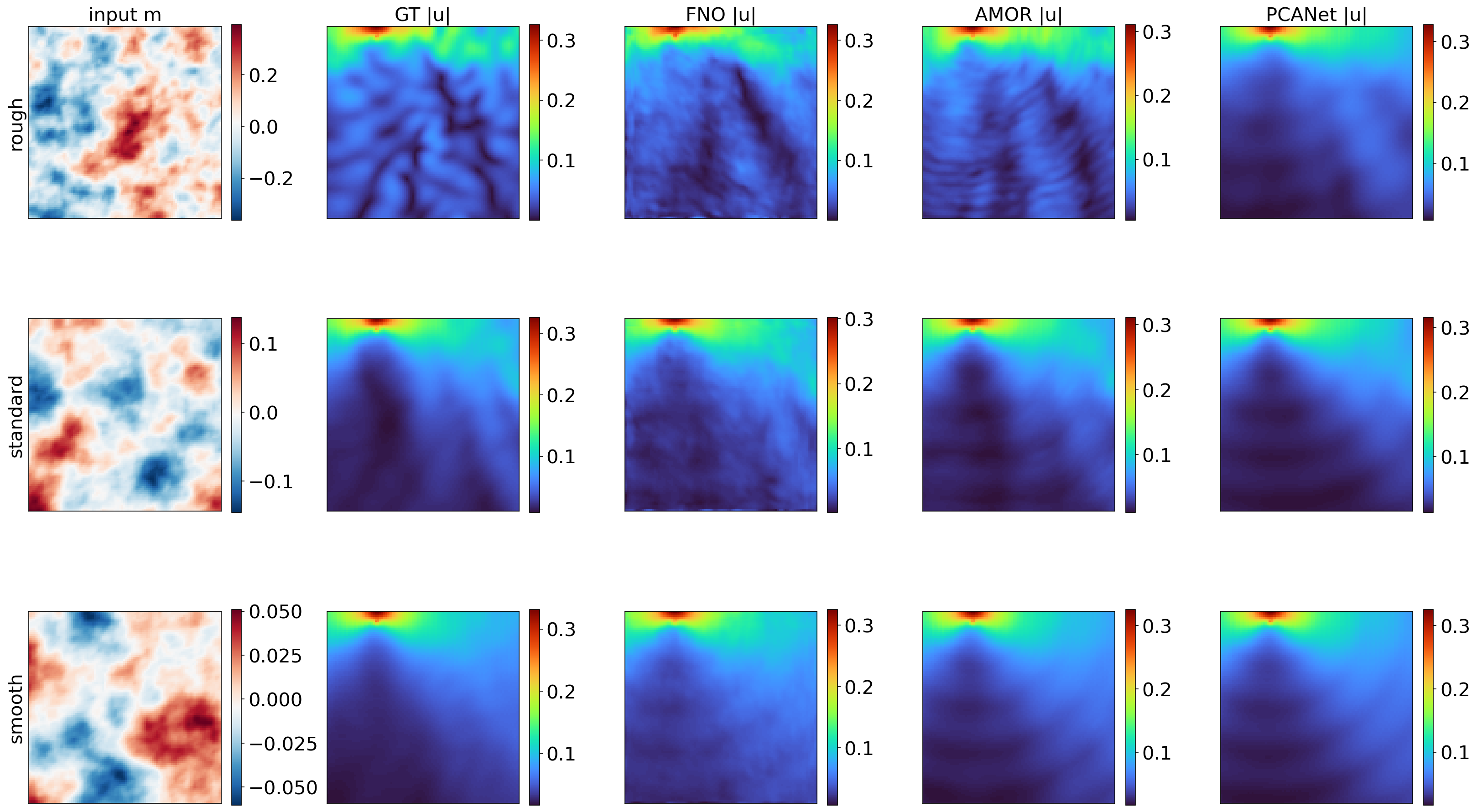}
    \caption{Predictions for one randomly selected test sample from each
    of the standard, smooth, and ultrasmooth Helmholtz datasets.
    Columns show the input, ground truth, and predictions from FNO,
    AMOR, and PCA-Net, respectively.}
    \label{fig:helm_pred}
\end{figure}

\begin{figure}[htbp]
    \centering
    \includegraphics[width=\linewidth]{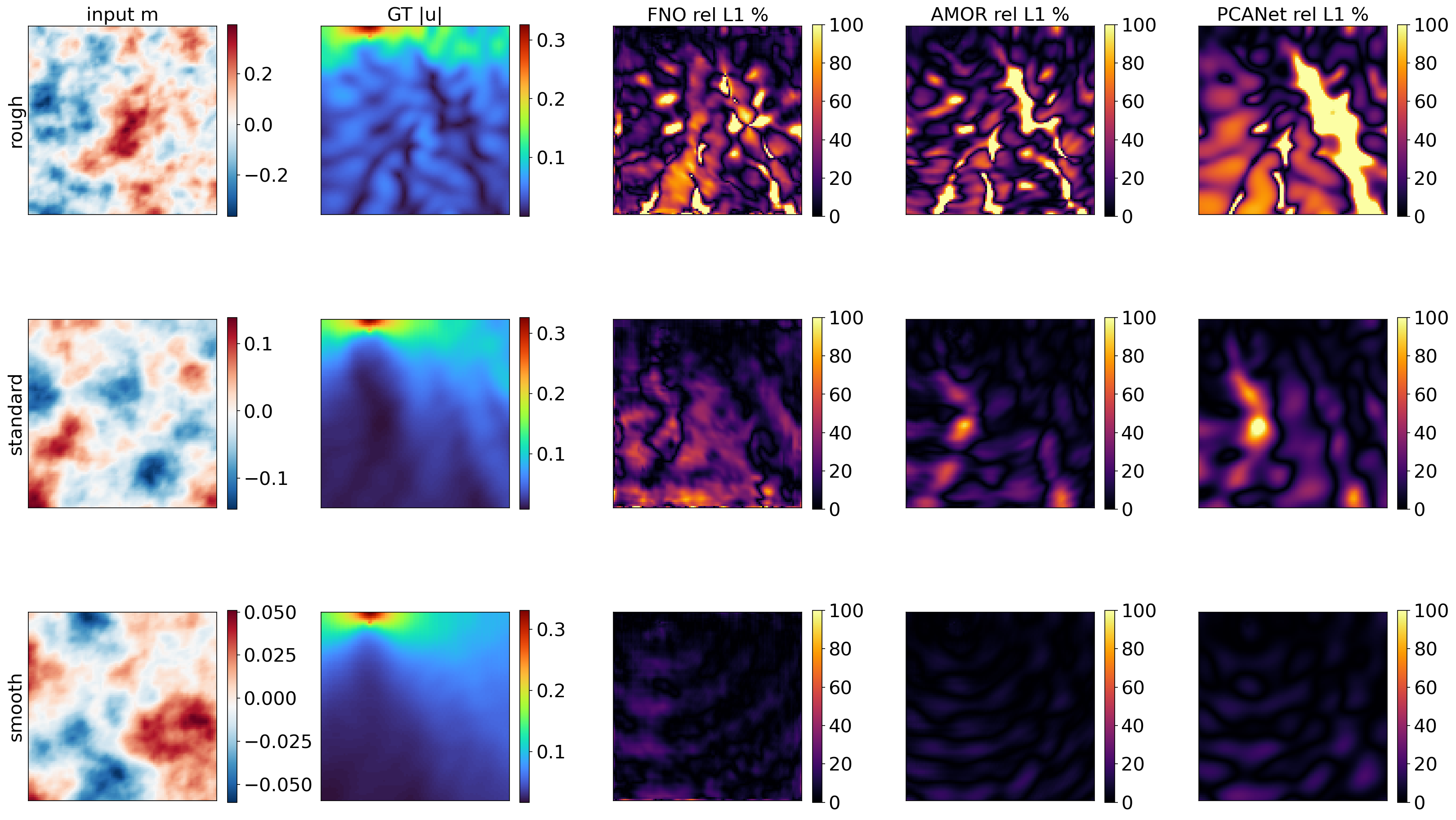}
    \caption{Pointwise relative absolute errors for one randomly selected
    test sample from each of the standard, smooth, and ultrasmooth
    Helmholtz datasets. The first two columns show the input and ground
    truth; the remaining columns show the errors of FNO, AMOR, and
    PCA-Net, respectively.}
    \label{fig:helm_err}
\end{figure}

\textbf{Kolmogorov Flow.}
Figure~\ref{fig:kolm_pred} compares the three models on one randomly
selected test sample from each Kolmogorov flow variant.
Figure~\ref{fig:kolm_err} displays their pointwise relative absolute
errors to complement the field predictions with a local view of accuracy.

\begin{figure}[htbp]
    \centering
    \includegraphics[width=\linewidth]{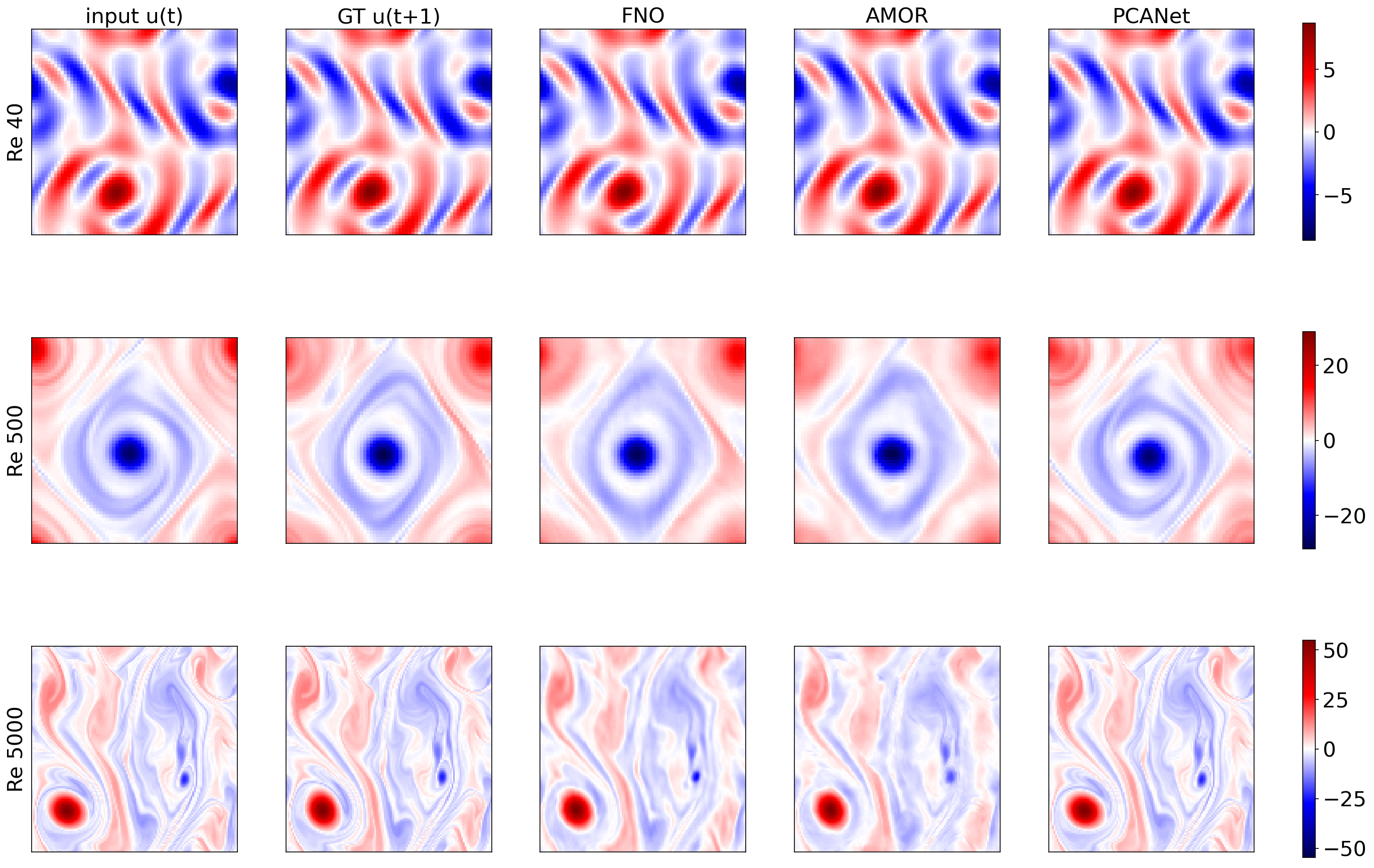}
    \caption{Predictions for one randomly selected test sample from each
    of the three Kolmogorov flow datasets. Columns show the input, ground
    truth, and predictions from FNO, AMOR, and PCA-Net, respectively.}
    \label{fig:kolm_pred}
\end{figure}

\begin{figure}[htbp]
    \centering
    \includegraphics[width=\linewidth]{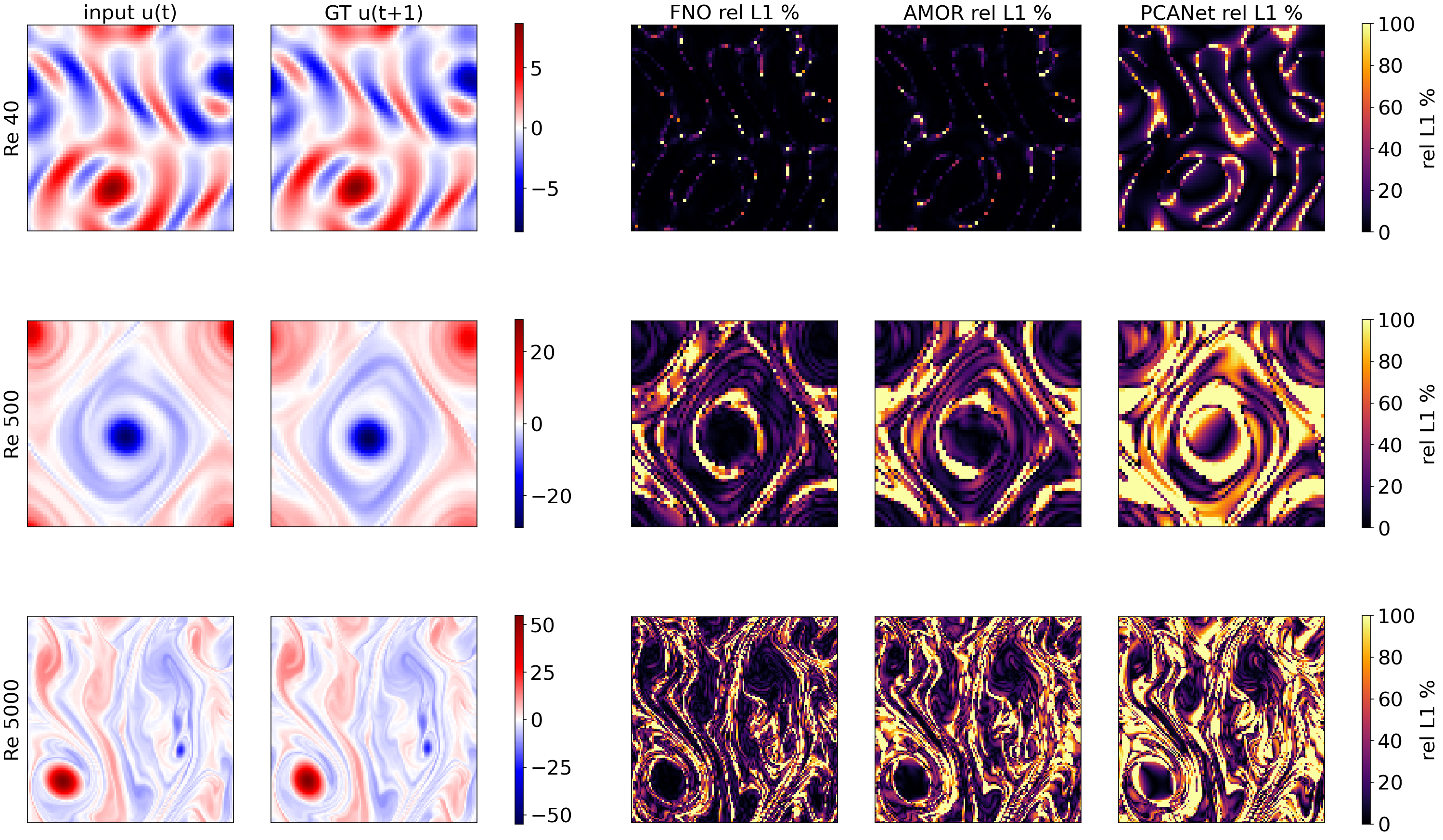}
    \caption{Pointwise relative absolute errors for one randomly selected
    test sample from each of the three Kolmogorov flow datasets.
    The first two columns show the input and ground truth; the remaining
    columns show the errors of FNO, AMOR, and PCA-Net, respectively.}
    \label{fig:kolm_err}
\end{figure}

\textbf{Airfoil and Elasticity.}
Figure~\ref{fig:airfoil_vis} compares AMOR predictions with the ground
truth and shows the pointwise relative absolute errors for two randomly
selected Airfoil test samples.
Figure~\ref{fig:elasticity_vis} provides the same comparison for two
randomly selected Elasticity test samples.

\begin{figure}[htbp]
    \centering
    \includegraphics[width=\linewidth]{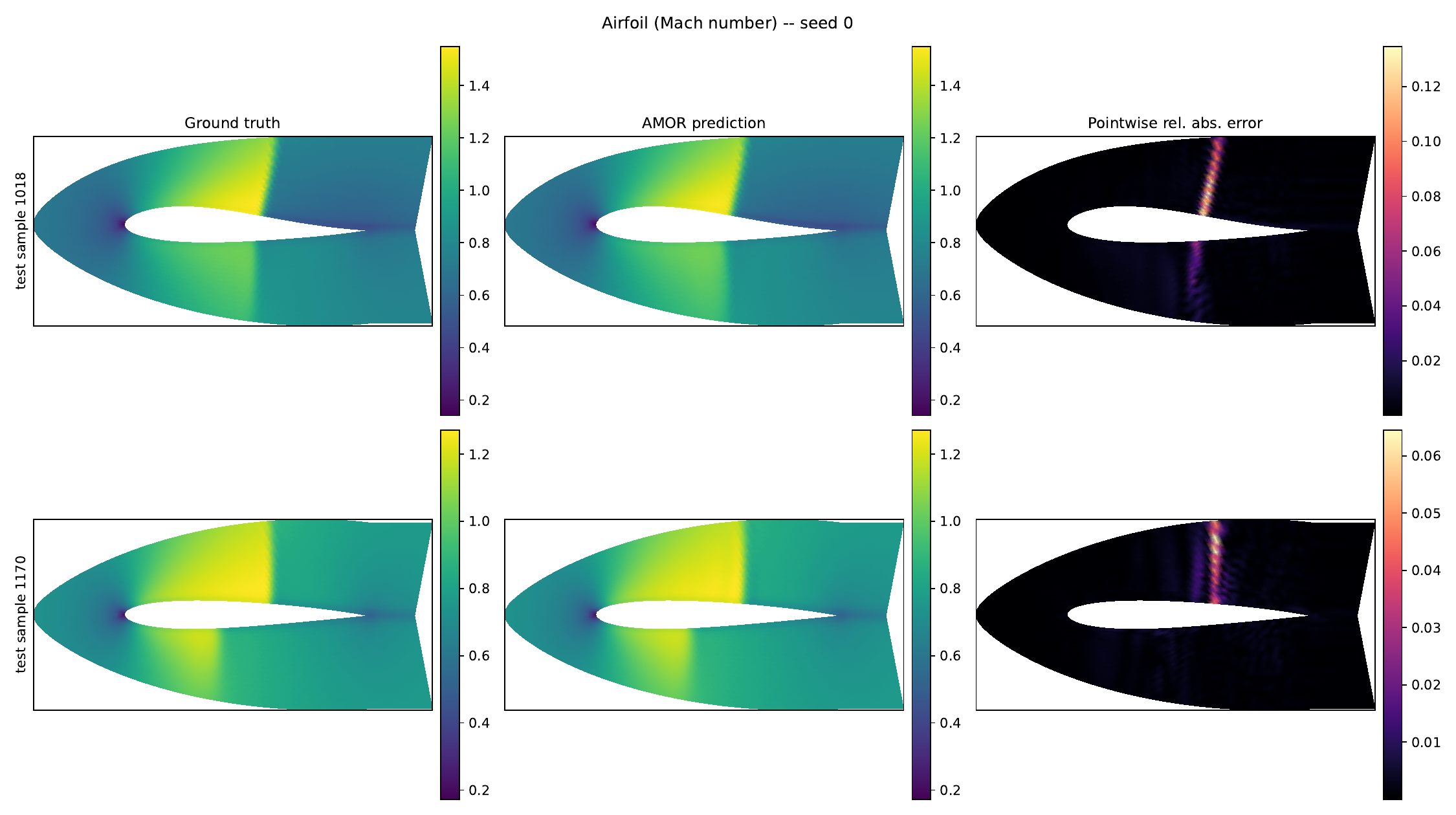}
    \caption{Ground truth, AMOR predictions, and pointwise relative
    absolute errors for two randomly selected test samples from the
    Airfoil dataset.}
    \label{fig:airfoil_vis}
\end{figure}

\begin{figure}[htbp]
    \centering
    \includegraphics[width=\linewidth]{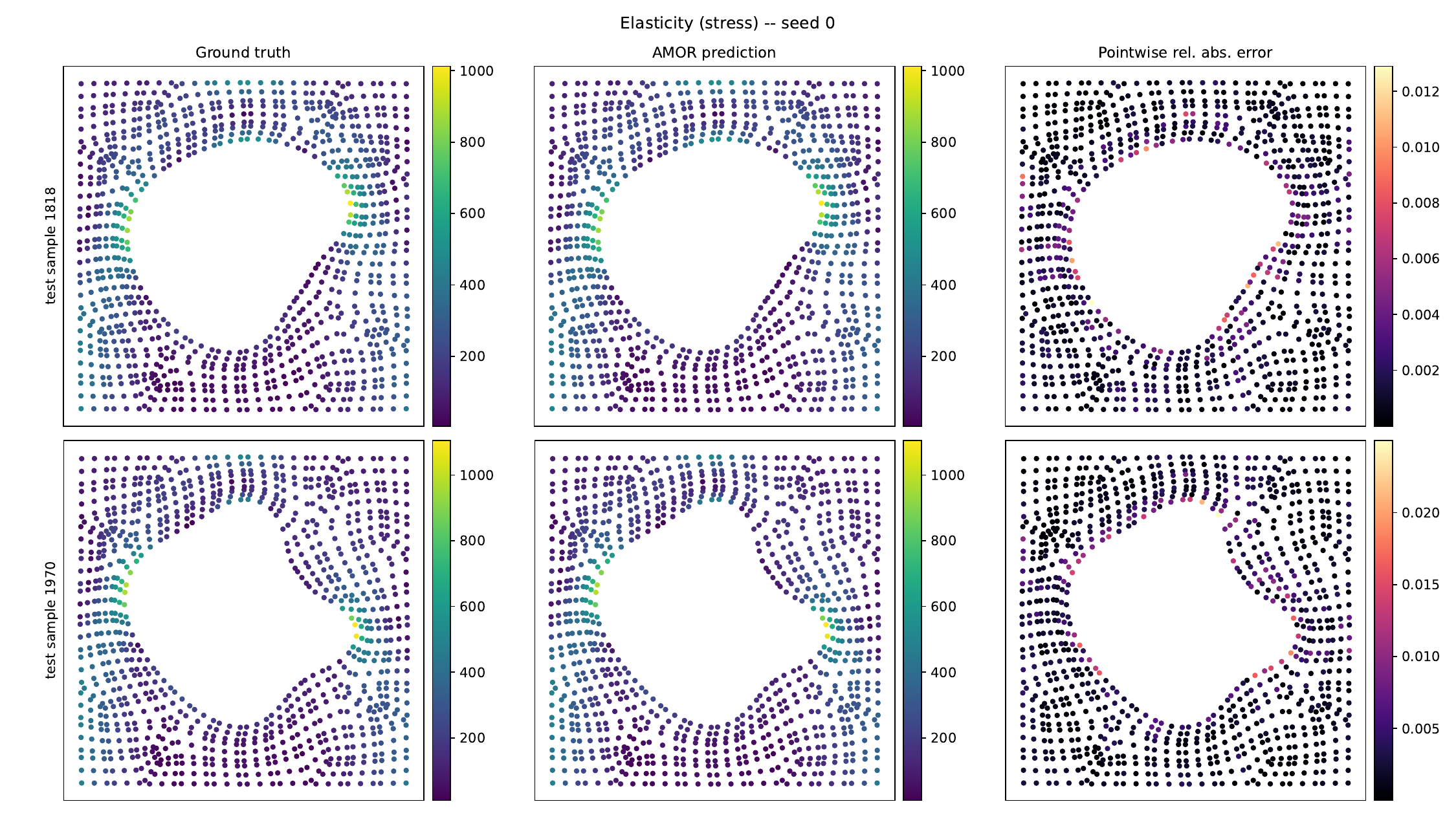}
    \caption{Ground truth, AMOR predictions, and pointwise relative
    absolute errors for two randomly selected test samples from the
    Elasticity dataset.}
    \label{fig:elasticity_vis}
\end{figure}

\subsection{Computational Scaling in Three Dimensions}\label{si:3d}

We measure inference time and training time per epoch on the Elastostatics
dataset while varying the mode counts and channel width.
All other hyperparameters and compute resources are held fixed within
each sweep.

\textbf{Inference at Matched Reconstruction Error.}
Figure~\ref{fig:scaling_matched} compares inference times with AMOR's
input and output basis sizes chosen separately to match the
Fourier-truncation reconstruction errors on the training inputs and
outputs.
Panel~(a) varies the FNO mode count per axis and the matched AMOR basis
sizes, while panel~(b) varies the channel width.

\begin{figure}[htbp]
    \centering
    \includegraphics[width=\linewidth]{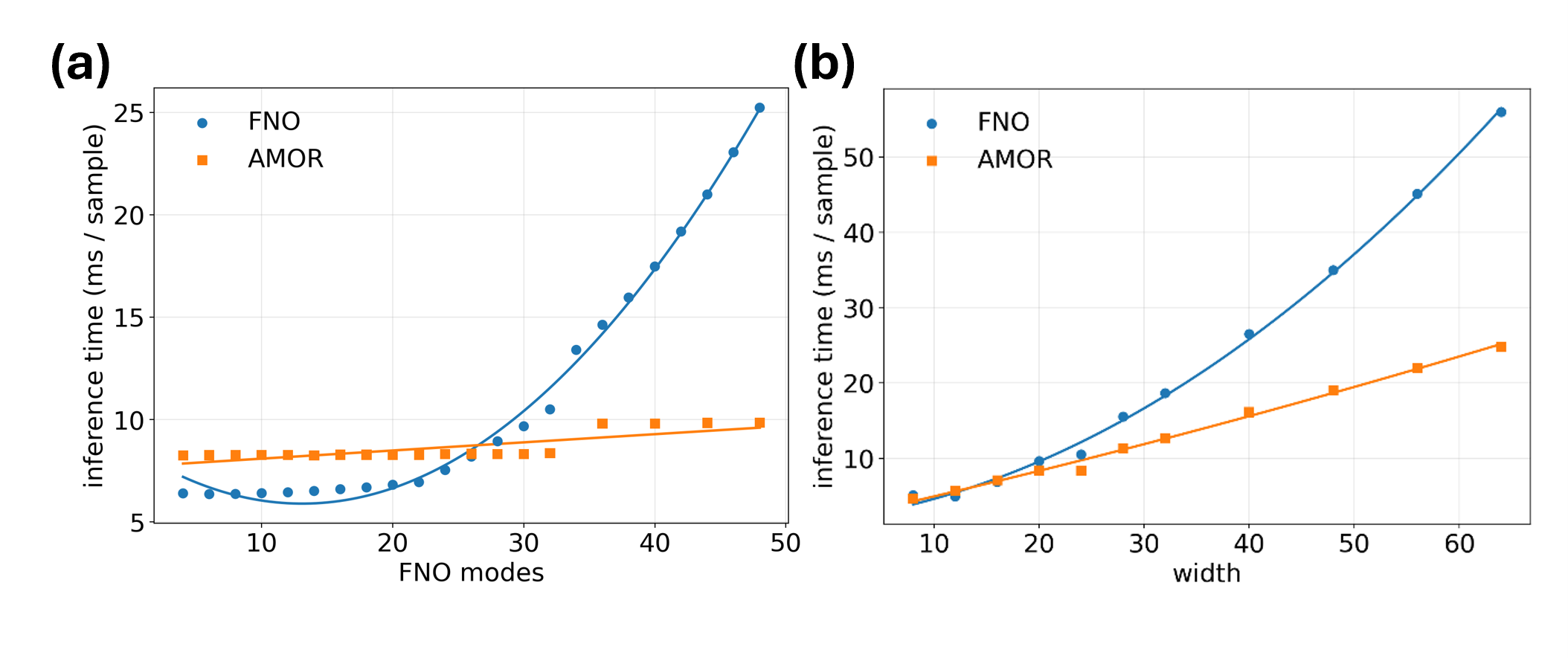}
    \caption{Inference-time scaling on the Elastostatics dataset.
    (a) FNO uses $m=m_x=m_y=m_z$ modes per axis; AMOR's input and output
    basis sizes are chosen separately to match the corresponding
    reconstruction errors on the training sets.
    (b) Channel-width sweep for FNO and AMOR.
    Other hyperparameters and compute resources are fixed within each
    sweep. Cyan and orange curves show empirical polynomial fits for
    FNO and AMOR, respectively, over the tested parameter ranges.}
    \label{fig:scaling_matched}
    \label{fig:scaling_mathced} 
\end{figure}

\textbf{Training Time.}
Figure~\ref{fig:scaling_train} reports training time per epoch as the
mode counts and channel width vary.
In the mode-count sweep, $m$ counts Fourier modes per axis for FNO,
whereas $k$ counts the basis functions in each of AMOR's input and
output bases.

\begin{figure}[htbp]
    \centering
    \includegraphics[width=\linewidth]{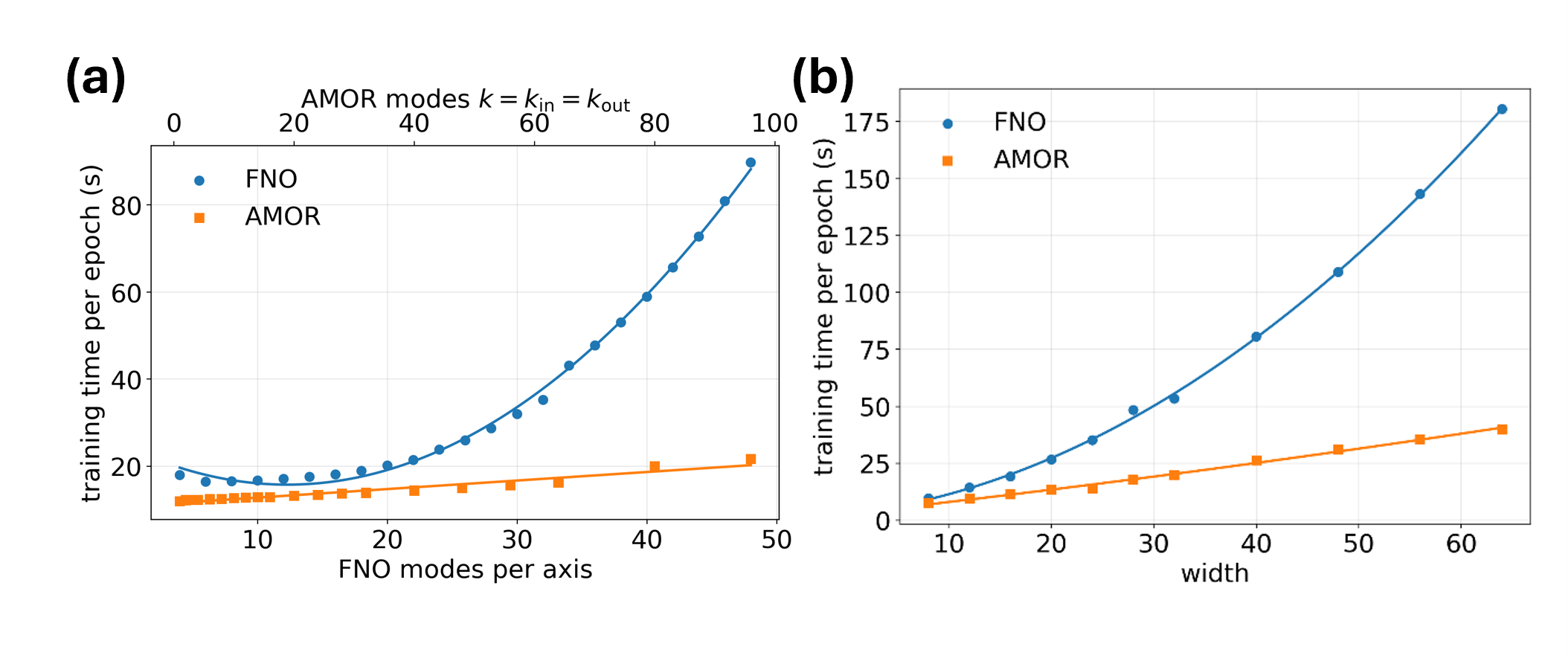}
    \caption{Training-time scaling on the Elastostatics dataset.
    (a) Mode-count sweep using $m=m_x=m_y=m_z$ for FNO and
    $k=k_{\mathrm{in}}=k_{\mathrm{out}}$ for AMOR.
    (b) Channel-width sweep for both models.
    Other hyperparameters and compute resources are fixed within each
    sweep. Cyan and orange curves show empirical polynomial fits for
    FNO and AMOR, respectively, over the tested parameter ranges.}
    \label{fig:scaling_train}
\end{figure}

\subsection{Aircraft Rib}\label{si:rib}

Figure~\ref{fig:rib_pred} shows AMOR's prediction and pointwise error compared to the ground truth on all three cases.
\begin{figure}[htbp]
    \centering
    \includegraphics[width=\linewidth]{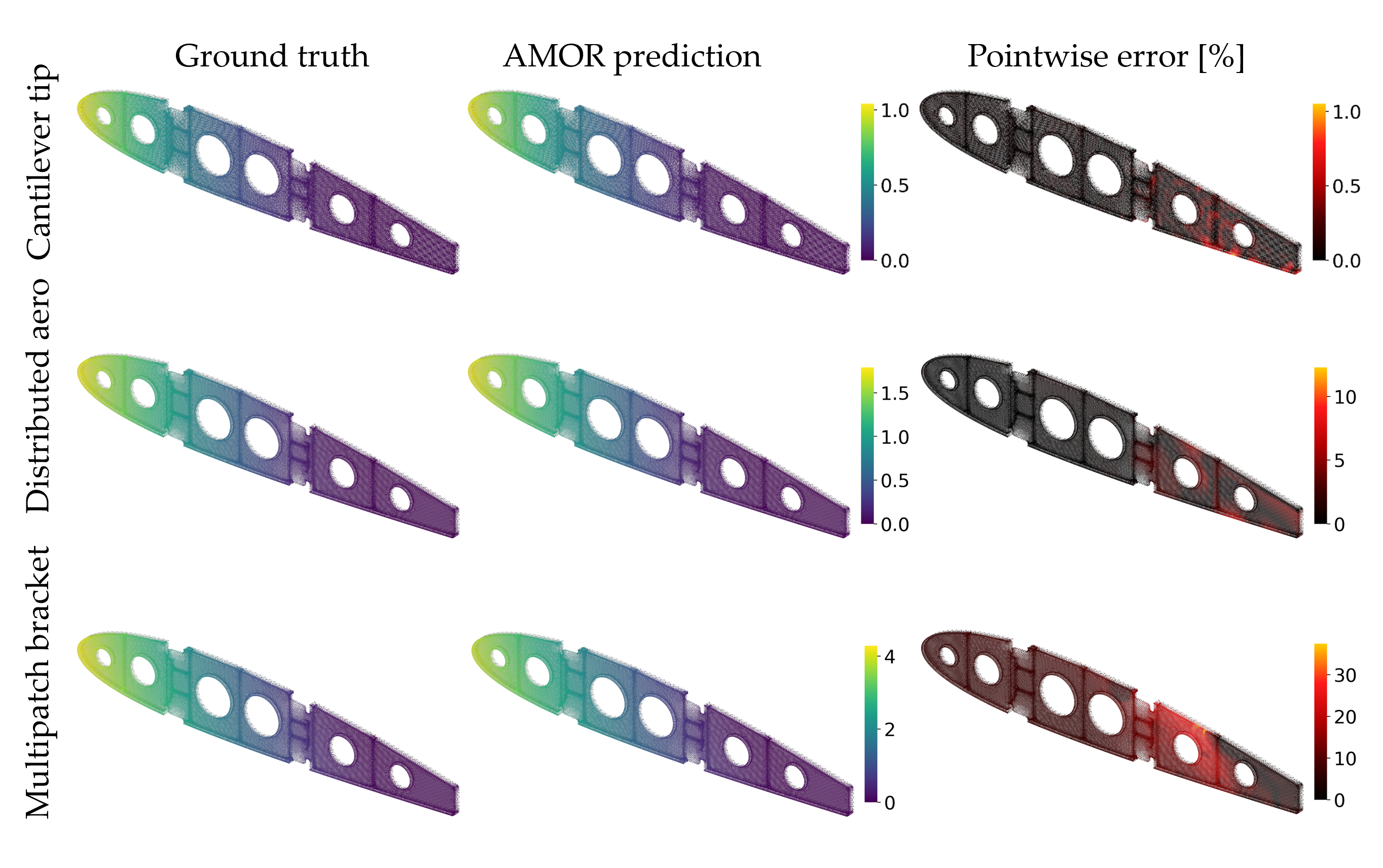}
    \caption{2D visualization of the displacement field magnitude $|u|$ ground truth compared to AMOR's prediction alongside the pointwise relative absolute error between the two for one random test split sample for each load scenario. AMOR captures the displacement field well in each case with the error increasing with the complexity of the load scenario.}
    \label{fig:rib_pred}
\end{figure}

\subsubsection{Basis comparison on cantilever tip dataset}\label{si:rib_basis}
Table~\ref{tab:rib_basis} compares the bases on the cantilever-tip dataset in the earlier tip configuration.
\begin{table}[htbp]
  \centering
  \caption{AMOR basis comparison on the aircraft-rib cantilever-tip dataset. Best results in \textbf{bold}.}
  \label{tab:rib_basis}
  \begin{tabular}{@{}lcc@{}}
    \toprule
    Basis & Rel.\ $L^1$ [\%] & Rel.\ $L^2$ [\%] \\
    \midrule
    PCA            & 0.0515 & 0.0478 \\
    Laplacian      & 0.0472 & 0.0502 \\
    Learned (Cayley) & \textbf{0.0369} & \textbf{0.0372} \\
    Learned (QR)     & 0.0529 & 0.0573 \\
    Jacobi         & 0.0767 & 0.0683 \\
    Chebyshev      & 0.0763 & 0.0676 \\
    Legendre       & 0.0786 & 0.0713 \\
    \bottomrule
  \end{tabular}
\end{table}

\subsubsection{Combined Model on mixed dataset}
A single model trained on the pooled $40\,000$-sample mixed dataset, using one shared basis and network across all three regimes, reaches an overall relative $L^1$/$L^2$ error of $6.57\%/7.34\%$.
Within each originating scenario, the pooled model is less accurate than the corresponding specialized model.
Table~\ref{tab:rib_mixed} breaks this error down by originating scenario.
\begin{table}[htbp]
\centering
\caption{Error of the single mixed-dataset model
(Table~\ref{tab:rib}, last row), broken down by originating loading
scenario.}
\label{tab:rib_mixed}
\begin{tabular}{lcc}
\toprule
Scenario (within mixed model) & Rel.\ $L^1$ (\%) & Rel.\ $L^2$ (\%) \\
\midrule
Distributed aero    & 2.27  & 2.27  \\
Multipatch bracket  & 11.93 & 13.52 \\
Cantilever tip      & 0.57  & 0.53  \\
\bottomrule
\end{tabular}
\end{table}
Relative to the scenario-specific models, error increases for all three cases under joint training, most sharply for tip loading ($0.06\%\to0.57\%$) and more moderately for aero ($1.17\%\to2.27\%$) and multipatch ($10.40\%\to11.93\%$). 
For the cantilever tip loading scenario these numbers are not directly comparable however as in the combined dataset the mapping is different and includes the load as an input field.
This is consistent with tip loading benefiting most from a large, scenario-tailored mass-weighted-PCA basis ($k{=}100/30$) that a single, more compact shared basis cannot fully reproduce, while multipatch loading -- already the dominant error source -- is comparatively less affected by joint training. 
Overall, a single AMOR model generalizes across the three physically distinct loading conditions on this mesh, at a small accuracy cost relative to specialized per-scenario models, with multipatch bracket loading remaining the main bottleneck in both settings.

\subsubsection{Mesh transfer (cantilever tip)}\label{si:rib_mesh_transfer}
We transfer the tip AMOR model trained on the coarse mesh (\texttt{h15}; \(k_{\mathrm{in}}{=}100\), \(k_{\mathrm{out}}{=}30\), mass-weighted PCA) to finer P1 meshes \texttt{h11}, \texttt{h7}, and \texttt{h5} by prolonging the coarse mass-PCA modes with the same P1 operator \(P\) used for nodal fields and keeping the network weights frozen.
The tip model uses four blocks of width $16$ with factorized coefficient tensors and shared bases ($3.1\times10^5$ parameters), and its bases are fitted in the consistent P1 mass inner product.
Evaluation uses \(100\) paired easy-GP samples with a common random seed, a shared geometric bounding box, and a fixed tip load (\(F{=}194.17\)).
The reference on each target mesh \(\mathcal{T}\) is the native fine-mesh FEM displacement \(u_{\mathcal{T}}\).

As a solver baseline on the same pairs we prolong the coarse FEM solution, \(P u_{15}\), and compare it to \(u_{\mathcal{T}}\) in the same relative \(L^1\)/\(L^2\) norms.
This baseline isolates the change of discrete space (and load representation) between mesh sizes; it is not a comparison against a single finest-mesh reference as in a classical mesh-convergence study.
\begin{table}[htbp]
\centering
\caption{Coarse-to-fine transfer on the cantilever-tip rib (relative
errors [\%], mean over \(100\) paired samples).
FEM: \(P u_{15}\) vs.\ native \(u_{\mathcal{T}}\).
AMOR: prolonged mass-PCA decoder with dual encoder, frozen \(\theta\).}
\label{tab:rib_mesh_transfer}
\begin{tabular}{@{}lcccc@{}}
\toprule
Mesh &
\multicolumn{2}{c}{FEM \(P u_{15}\)} &
\multicolumn{2}{c}{AMOR (dual)} \\
\cmidrule(lr){2-3} \cmidrule(lr){4-5}
& Rel.\ \(L^1\) & Rel.\ \(L^2\) & Rel.\ \(L^1\) & Rel.\ \(L^2\) \\
\midrule
\texttt{h11} & 9.79 & 10.11 & 9.45 & 9.64 \\
\texttt{h7}  & 16.34 & 17.07 & 12.03 & 10.70 \\
\texttt{h5}  & 17.16 & 17.91 & 15.02 & 14.34 \\
\bottomrule
\end{tabular}
\end{table}
Both errors increase from \texttt{h11} to \texttt{h5}, consistent with a growing coarse--fine discrepancy.
AMOR improves on the coarse-solution prolongation baseline on all three meshes, by $0.3$--$4.3$ percentage points in relative $L^1$ error.

\end{document}